\documentclass[12pt]{amsart}
\usepackage{srcltx}
\usepackage{hyperref}

\newcounter{ourcount}
\advance\textheight\topskip
\usepackage{amsmath,amsfonts,amsthm,amssymb,amscd}
\allowdisplaybreaks

 \usepackage{graphicx}
 \usepackage[curve,matrix,arrow,color]{xy}

 \usepackage[all,color]{xypic}
 \usepackage[usenames,dvipsnames]{color}
 \xyoption{line}

\usepackage[mathscr]{eucal}
\usepackage[OT2,T1]{fontenc}

\newcommand{\bref}[1]{\textbf{\ref{#1}}}

\newcommand{\half}{%
  \mathchoice{\ffrac{1}{2}}{\frac{1}{2}}{\frac{1}{2}}{\frac{1}{2}}}
\newcommand{\rmi}{\mathrm{i}}
\newcommand{\id}{\mathrm{id}}

\newcommand{\Walg}{\mathcal{W}}

\newcommand{\SLiiZ}{SL(2,\oZ)}
\newcommand{\cZ}{\mathsf{Z}}
\newcommand{\modS}{\mathscr{S}}
\newcommand{\modT}{\mathscr{T}}
\newcommand{\repLy}{\pi}

\newcommand{\Q}{\mathsf{Q}}
\newcommand{\Salg}{\mathsf{S}}

\newcommand{\B}{\mathsf{B}}
\newcommand{\Bel}{\mathsf{b}}
\newcommand{\Cel}{\mathsf{c}}

\newcommand{\z}{\mathsf{z}}

\newcommand{\copS}{\Delta^{\Salg}}

\newcommand{\cat}{\mathcal{C}}
\newcommand{\catD}{\mathcal{D}}
\newcommand{\catSF}{\mathcal{S\hspace{-1.2pt}F}}
\newcommand{\svect}{\mathrm{{\bf Svect}}}
\newcommand{\vect}{\mathrm{{\bf vect}}}

\newcommand{\rep}{\mathrm{\bf Rep}\,}

\newcommand{\repsv}{\mathrm{\bf Rep}_{\mathrm{s.v.}}}
\newcommand{\repQ}{\rep\Q}
\newcommand{\repS}{\rep\Salg}

\newcommand{\fun}{\mathcal{F}}    
\newcommand{\funE}{\mathcal{E}}   

\newcommand{\funD}{\mathcal{D}}   
\newcommand{\funCS}{\funD}        
\newcommand{\funSQ}{\mathcal{G}}  

\newcommand{\isoD}{\Delta}
\newcommand{\isoG}{\Gamma}
\newcommand{\del}{\delta}  

\newcommand{\tensor}{\otimes}

\newcommand{\svtensor}{\tensor_{\svect}}
\newcommand{\Ctensor}{*}
\newcommand{\Qtensor}{\tensor_{\repQ}}
\newcommand{\Stensor}{\tensor_{\repS}}

\newcommand{\as}{\Phi}      
\newcommand{\Sas}{\Lambda}  

\newcommand{\assocQ}{\alpha^{\repQ}} 
\newcommand{\assocS}{\alpha^{\repS}} 

\newcommand{\sflip}{\tau^{\mathrm{s.v.}}}   

\newcommand{\brC}{c}               
\newcommand{\brQ}{\sigma}          

\newcommand{\Rst}{R^{\mathrm{(st.)}}}   
\newcommand{\Rsf}{\mathcal{R}}          

\newcommand{\Salpha}{\boldsymbol{\alpha}}
\newcommand{\Sbeta}{\boldsymbol{\beta}}
\newcommand{\Sgamma}{\boldsymbol{\gamma}}

\newcommand{\cchi}{\boldsymbol{\chi}}
\newcommand{\rrho}{\boldsymbol{\rho}}
\newcommand{\bkappa}{\pmb{\boldsymbol{\varkappa}}}
\newcommand{\vvarphi}{\pmb{\varphi}}

\newcommand{\xp}{\mathsf{x}^+}
\newcommand{\xm}{\mathsf{x}^-}
\newcommand{\xpm}{\mathsf{x}^{\pm}}
\newcommand{\xmp}{\mathsf{x}^{\mp}}
\newcommand{\LL}{\mathsf{L}}

\newcommand{\q}{\mathfrak{q}}
\newcommand{\ffrac}[2]{\mbox{\footnotesize$\displaystyle\frac{#1}{#2}$}}

\newcommand{\one}{\boldsymbol{1}}

\newcommand{\UresSL}[1]{\overline{U}_{\q} s\ell(#1)}
\newcommand{\iUresSL}[1]{\overline{U}_{\rmi} s\ell(#1)}

\newcommand{\oC}{\mathbb{C}}

\newcommand{\K}{\mathsf{K}}
\newcommand{\F}{\mathsf{F}}

\newcommand{\E}{\mathsf{E}}

\newcommand{\ff}{\mathsf{f}}

\newcommand{\algGr}{\mathsf{G}}

\newcommand{\XX}{\mathsf{X}} 

\newcommand{\PP}{\mathsf{P}}

\newcommand{\qtr}{\mathrm{qTr}}

\newcommand{\ad}{\mathrm{Ad}}
\newcommand{\Ch}{\mathfrak{Ch}}

\newcommand{\radmap}{\boldsymbol{\phi}}
\newcommand{\drmap}{\boldsymbol{\chi}}

\newcommand{\coint}{{\boldsymbol{c}}}
\newcommand{\rint}{{\boldsymbol{\mu}}}

\newcommand{\balance}{{\boldsymbol{g}}}
\newcommand{\ribbon}{{\boldsymbol{v}}}
\newcommand{\sqs}{{\boldsymbol{u}}}

\newcommand{\cas}{\boldsymbol{C}}

\newcommand{\idem}{\boldsymbol{e}}
\newcommand{\nilp}{\boldsymbol{w}}

\newcommand{\oZ}{\mathbb{Z}}

\swapnumbers
\newtheorem{Thm}[subsection]{Theorem}

\newtheorem{Lemma}[subsection]{Lemma}

\newtheorem{Prop}[subsection]{Proposition}

\newtheorem{Cor}[subsection]{Corollary}

\theoremstyle{definition}
\newtheorem{Dfn}[subsection]{Definition}

\newtheorem{Rem}[subsection]{Remark}

\begin{document}
\title[quasi-Hopf algebra of symplectic fermions]{%
Symplectic fermions and \\
a quasi-Hopf algebra structure on $\iUresSL2$}

\author{A.M.~Gainutdinov, I.~Runkel}
\address{AMG: Fachbereich Mathematik, Universit\"at Hamburg, Bundesstra\ss e 55,
20146 Hamburg \\
and  DESY, Theory Group, Notkestrasse 85, Bldg. 2a, 22603 Hamburg, Germany}
\email{azat.gainutdinov@uni-hamburg.de}

\address{IR: Fachbereich Mathematik, Universit\"at Hamburg, Bundesstra\ss e 55,
20146 Hamburg, Germany}
\email{ingo.runkel@uni-hamburg.de}

\begin{abstract}
We consider the (finite-dimensional) restricted quantum group $\UresSL2$ at $\q=\rmi$. We show that $\iUresSL2$ does not allow for a	universal
R-matrix, even though $U \tensor V \cong V \tensor U$ holds for all finite-dimensional representations $U,V$ of $\iUresSL2$. We then give an explicit coassociator $\as$ and a universal R-matrix $R$ such that $\iUresSL2$ becomes a quasi-triangular quasi-Hopf algebra.

Our construction is motivated by the two-dimensional chiral conformal field theory of symplectic fermions with central charge $c=-2$. There, a braided monoidal category, $\catSF$, has been computed from the factorisation and monodromy properties of conformal blocks,
and we prove 
 that $\rep(\iUresSL2,\as,R)$ is braided monoidally equivalent to $\catSF$.
\end{abstract}

\maketitle

\setcounter{tocdepth}{1}
\tableofcontents
    
\thispagestyle{empty}
\newpage
    
\section{Introduction}

Recall the definition of the 
\textit{restricted}\footnote{This algebra was named ``restricted'' in~\cite{[FGST]} and this name has not to be confused with the ``restricted form'' of $U_{\q} s\ell(2)$ -- 
Lusztig's integral form, see e.g.~\cite{ChPr}.}
quantum group $\UresSL2$, with $\q = e^{\rmi\pi/p}$ and $p$ is positive integer~\cite{ChPr,[FGST]}.
 It has the generators $\E$, $\F$, and
$\K^{\pm1}$ satisfying the standard relations for the quantum $s\ell(2)$,
\begin{equation}\label{Uq-relations}
  \K\E\K^{-1}=\q^2\E\ ,\quad
  \K\F\K^{-1}=\q^{-2}\F\ ,\quad
  [\E,\F]=\ffrac{\K-\K^{-1}}{\q-\q^{-1}}\ ,
\end{equation}
and the  additional relations
\begin{equation}\label{root-rel}
  \E^{p}=\F^{p}=0\ ,\quad \K^{2p}=\one \ .
\end{equation}
The comultiplication, counit and antipode are given by
\begin{align}
\label{eq:antipode}
  \Delta(\E)&=\one\otimes \E+\E\otimes \K\ ,\quad &
  \Delta(\F)&=\K^{-1}\otimes \F+\F\otimes\one\ ,\quad &
  \Delta(\K)&=\K\otimes \K\ .
  \nonumber\\ 
\epsilon(\E)&=0 \ ,& \epsilon(\F)&=0\ ,\quad & \epsilon(\K)&=1\ ,\nonumber
\\
S(\E)&=-\E\K^{-1}\ ,\quad& S(\F)&=-\K\F\ ,\quad& S(\K)&=\K^{-1}\ .
\end{align}
This defines a Hopf algebra of dimension $\mathrm{dim}\,\UresSL2 = 2p^3$.

There is a close relation~\cite{Feigin:2005xs} between the category
	$\rep\UresSL2$
of finite dimensional representations of $\UresSL2$ and the category $\rep\Walg_p$ of modules of the 
$W_{1,p}$ triplet
vertex operator algebra 
$\Walg_p$~\cite{[K-first],Gaberdiel:1996np,Fuchs:2003yu,Carqueville:2005nu,[AM]} 
which occurs in logarithmic rational conformal field theory. It is known that
\begin{itemize}		
		\item for $p \ge 2$, $\rep\UresSL2$ is equivalent to $\rep\Walg_p$
	as a $\oC$-linear category \cite{[AM],Nagatomo:2009xp},
		\item for $p \ge 3$, $\rep\UresSL2$ is \textit{not} equivalent to $\rep\Walg_p$ as
	a braided tensor category, or even only as a tensor category.
\end{itemize}
The second point follows as $\rep\Walg_p$ is
braided~\cite{[HLZ],Huang:2013jza,Tsuchiya:2012ru},
but $\rep\UresSL2$ is not braidable since there are 
	finite-dimensional 
$\UresSL2$-representations $U,V$ such that $U \tensor V$ is not isomorphic to $V \tensor U$~\cite{KS}.

The situation for $p=2$ is special: In this case $\q=\rmi$ and for all finite-dimensional modules $U$, $V$ over $\iUresSL2$ we have isomorphisms $U\tensor V \cong V\tensor U$, see~\cite{KS}. However, we will show that
it is not possible to chose a natural family of such isomorphisms that satisfy the hexagon condition for the braiding:

\begin{Thm}\label{thm:1}
The category $\rep\iUresSL2$  is not braidable, or, equivalently, $\iUresSL2$ has no universal R-matrix.
\end{Thm}

	In the following, we will use the term ``R-matrix'' instead of ``universal R-matrix''.
	
	\medskip

	By Theorem~\bref{thm:1}, the category $\rep\iUresSL2$ cannot be tensor equivalent to $\rep\Walg_2$, as the latter is braided and the former not braidable.
	On the other hand, if we divide $\iUresSL2$ by the ideal generated by $\K^2-\one$, there is an	 R-matrix~\cite{Lus90,RT91,KM91} (see also~\cite[Thm.\,IX.7.1]{Kassel})
with the standard form
\begin{equation}\label{eq:R00-EF}
\Rst = \half\sum_{m,n=0,1}(-1)^{mn}\K^m\tensor\K^n\bigl(\one\tensor\one + 2\rmi \E\tensor\F\bigr) \ .
\end{equation}
In~\cite{[LN]}, R-matrices of this form were classified in 
restricted-type
 quantum groups for finite-dimensional simple complex Lie algebras with different quotients in the Cartan part. For  $\iUresSL2$, \cite{[LN]} indeed find no solution, and our Theorem~\bref{thm:1} extends this negative result from R-matrices of the form~\eqref{eq:R00-EF} to all R-matrices.

The
motivation behind the research presented in this paper was to find a suitable small modification of $\iUresSL2$ to make its representations agree with those of $\Walg_2$ as a braided tensor category.
We find that it is possible to define a {\sl quasi}-Hopf structure on $\iUresSL2$ (with the same algebraic relations, coproduct, counit and antipode) that makes the algebra quasi-triangular and extends the quasi-triangular Hopf-structure \eqref{eq:R00-EF} from the quotient algebra. We will now describe this quasi-triangular quasi-Hopf structure and then comment on the relation to $\rep\Walg_2$. Our conventions on quasi-Hopf algebras are collected in Appendix~\bref{app:qausi-Hopf}.

	The quasi-Hopf and quasi-triangular structure depend on a parameter 
\begin{equation}\label{eq:beta-param}	
	\beta \in \oC \quad \text{which satisfies} \quad  \beta^4=-1 \ .
\end{equation}	
Define the central idempotents
\begin{equation}\label{idem-01}
	\idem_0=\half(\one+\K^2) \ , \qquad  \idem_1=\half(\one-\K^2) \ .
\end{equation}
The coassociator $\as$, an invertible element in $(\iUresSL2)^{\otimes 3}$,
can be written component-wise as
\begin{align}\label{eq:as-intro}
\as = & ~ \idem_0\tensor\idem_0\tensor\idem_0 \\
&
+ \idem_1\tensor\idem_0\tensor\idem_0
+ \as^{010}\idem_0\tensor\idem_1\tensor\idem_0
+ \idem_0\tensor\idem_0\tensor\idem_1
\nonumber\\
& + \idem_1\tensor\idem_1\tensor\idem_0
  + \as^{101}\idem_1\tensor\idem_0\tensor\idem_1
  +\idem_0\tensor\idem_1\tensor\idem_1
\nonumber\\
& + \as^{111}\idem_1\tensor\idem_1\tensor\idem_1 \ ,
\nonumber
\end{align}
where, in a PBW basis with order $\E,\F,\K$,
\begin{align*}
\as^{010} =~& \one\tensor\one\tensor\one
    - (1 + \rmi) \E\tensor\K\tensor\F + (1 - \rmi) \F\K\tensor\K\tensor\E\K +
  2 \E\F\K\tensor\one\tensor\E\F\K\ ,
\\
\as^{101} =~& (\rmi -1) \one\tensor\F\tensor\E\K +
 \one\tensor\K\tensor\one +
 2 \one\tensor\E\F\tensor\E\F\K
 - (1 + \rmi)\one\tensor\E\K\tensor\F
\\
&- (1 + \rmi) \E\tensor\F\K\tensor\one +
 2 \rmi \E\tensor\E\F\K\tensor\F + (1 - \rmi)\F\K\tensor\E\tensor\one
\\
& -
 2 \rmi \F\K\tensor\E\F\K\tensor\E\K -
 2 \E\F\K\tensor\E\F\tensor\one\ ,
\\
\as^{111} =~& \rmi\beta^2\bigl\{ \, (1 + \rmi)\one\tensor\F\tensor\E\K -
 \one\tensor\K\tensor\one -
 2 \one\tensor\E\F\tensor\one -
 2 \rmi \one\tensor\E\F\tensor\E\F\K
\\
&  + (1 + \rmi) \one\tensor\E\K\tensor\F
  - (1 + \rmi) \E\tensor\one\tensor\F
  + (1 + \rmi) \E\tensor\F\K\tensor\one +
 2 \rmi \E\tensor\F\K\tensor\E\F\K
\\
&+
 2 \E\tensor\E\F\K\tensor\F
 - (1 + \rmi)\F\K\tensor\one\tensor\E\K
 + (1 + \rmi) \F\K\tensor\E\tensor\one +
 2 \rmi \F\K\tensor\E\tensor\E\F\K
\\
&+
 2 \F\K\tensor\E\F\K\tensor\E\K +
 2 \rmi \E\F\K\tensor\F\tensor\E\K
 + 2 \rmi \E\F\K\tensor\K\tensor\E\F\K -
 2 \rmi \E\F\K\tensor\E\F\tensor\one
\\
&
 + 2 \rmi \E\F\K\tensor\E\K\tensor\F 
 +4 \E\F\K\tensor\E\F\tensor\E\F\K
\, \bigr\}\ .
 \end{align*}
We take the same anti-automorphism $S$ on $\iUresSL2$ as antipode (see \eqref{eq:antipode}). We therefore have the same dual objects, but the duality maps get modified by evaluation and coevaluation elements $\Salpha$ and $\Sbeta$ (which are part of the definition of a quasi-Hopf algebra). We choose
\begin{equation}\label{eq:intro-al-be-def}
 \Salpha = \one\ ,\qquad \Sbeta =  \idem_0 -2\rmi\beta^2\cas\idem_1\ ,
\end{equation}
where $\cas := \F\E-\frac{\rmi}{4}(\K-\K^{-1})$ is the Casimir element.

Finally, we give an R-matrix
which extends the quasi-triangular structure~\eqref{eq:R00-EF} from the quotient by $\K^2-\one$ to the whole quantum group. We define an invertible element in $\iUresSL2\tensor\iUresSL2$ as
\begin{equation}\label{RQ}
R = R^{00}\idem_0\tensor\idem_0 + R^{01}\idem_0\tensor\idem_1
+ R^{10}\idem_1\tensor\idem_0 + R^{11}\idem_1\tensor\idem_1,
\end{equation}
with
\begin{align*}
R^{00} &= \half\bigl(\one\tensor\one + \one\tensor\K + \K\tensor\one - \K\tensor\K\bigr)\bigl\{\one\tensor\one +2\rmi\E\tensor\F\bigr\},
\\
R^{01} &= \half\bigl(\one\tensor\one -\rmi \one\tensor\K + \K\tensor\one + \rmi\K\tensor\K\bigr)\bigl\{\one\tensor\one
    +(1-\rmi)\F\K\tensor\E\K
\\
& \hspace{10em}
- (1-\rmi)\E\tensor\F +(1+\rmi)\E\F\K\tensor\one + 2\rmi\E\F\K\tensor\E\F\K\bigr\},
\\
R^{10} &= \half\bigl(\one\tensor\one + \one\tensor\K - \rmi \K\tensor\one + \rmi\K\tensor\K\bigr)
\bigl\{\one\tensor\K + (1 - \rmi) \F\K\tensor\E
\\
& \hspace{10em}
 + (1 - \rmi) \E\tensor\F\K  - (1 - \rmi) \one\tensor\E\F -
 2 \rmi \E\F\K\tensor\E\F\bigr\},\\
R^{11} &= \ffrac{\beta}{2}\bigl(\one\tensor\one -\rmi \one\tensor\K -\rmi \K\tensor\one + \K\tensor\K\bigr)
\bigl\{-\rmi \K\tensor\one + 2 \rmi \E\K\tensor\F
\\
& \hspace{10em}
+ (1 - \rmi)\K\tensor\E\F\K - (1 + \rmi)\E\F\tensor\one -
 2 \rmi \E\F\tensor\E\F\K
\bigr\}.
\end{align*}
We emphasise that $R^{00}$ here equals the standard 
	R-matrix
$\Rst$ in~\eqref{eq:R00-EF}.

\begin{Thm}\label{thm:2}
	The Hopf algebra $\iUresSL2$ becomes a quasi-triangular quasi-Hopf algebra when equipped with the coassociator $\as$, the R-matrix $R$, the evaluation element $\Salpha$ and the coevaluation element $\Sbeta$.
\end{Thm}
	Here, $\as$ modifies the associator in the category
	$\rep\iUresSL2$,
$R$ gives the braiding, and $\Salpha$ and $\Sbeta$ enter the
	definition of evaluation and coevaluation maps 
	(see	Appendix~\bref{app:qausi-Hopf} and	
\cite[Sec. 16.1]{ChPr}).

\begin{Cor}
The coassociator $\as$ defines a non-trivial cohomology class in the 3rd Hopf algebra cohomology. In particular, our quasi-Hopf algebra $(\iUresSL2, \as)$ cannot be obtained as a Drinfeld twist of the Hopf algebra $\iUresSL2$.
\end{Cor}

	For details on Hopf algebra cohomology, or what is left thereof in the non-abelian case, we refer to~\cite[Sec.~2]{Majid-book}
 -- we will not make further use of it in this paper.
The above corollary is equivalent to the statement that the identity functor on $\rep\iUresSL2$ cannot be endowed with a monoidal structure such that it becomes a monoidal functor from  $\rep(\iUresSL2, \as)$ -- where the associator is given by acting with $\as$ -- to $\rep\iUresSL2$. This is an immediate consequence of Theorems \bref{thm:1} and \bref{thm:2}, as $\rep(\iUresSL2, \as)$ is braidable while $\rep\iUresSL2$ is not.  Hence, there cannot be a monoidal equivalence at all, let alone a monoidal structure on the identity functor.

\vspace{1mm}

Next we comment on the relation to the vertex operator algebra $\Walg_2$.
At $p=2$, the vertex operator algebra $\Walg_p$ has a symplectic fermion
construction~\cite{Kausch:1995py,Abe:2005}
-- a chiral
rational logarithmic conformal field theory at central charge $c=-2$. In~\cite{Runkel:2012cf}, a braided tensor category
$\catSF$
of  the symplectic fermion fields was obtained from monodromy properties of 
conformal blocks (we review the category $\catSF$ in Section~\bref{sec:cat-C}). 
	In fact, due to a $\mathbb{Z}_2$-grading of the tensor product one naturally obtains four different braided monoidal categories $\catSF$ which we parametrise by the parameter $\beta$ from \eqref{eq:beta-param} already used for $\as$ and $R$.
The category obtained from symplectic fermion conformal blocks corresponds to $\beta=\mathrm{e}^{-\frac{\rmi\pi}{4}}$ in this convention. 
 Conjecturally, for this value of $\beta$, $\catSF$ is equivalent as a braided monoidal category to $\rep\Walg_2$,
 but the latter has not yet been computed explicitly.

\begin{Thm}\label{thm:3}
The categories  $\rep(\iUresSL2, \as, R)$ and $\catSF$ are equivalent as $\oC$-linear braided monoidal categories for each $\beta$ satisfying $\beta^4=-1$.
\end{Thm}

Theorem~\bref{thm:3} is our main result. 
 Under the conjectural braided monoidal equivalence between $\catSF$ and  $\rep\Walg_2$, Theorem~\bref{thm:3} is an extension 
 of
 the equivalence of the
 representation categories of $\iUresSL2$ and of $\Walg_2$ as
 $\oC$-linear categories
-- established in \cite{Feigin:2005xs,Nagatomo:2009xp} --
to an equivalence of braided monoidal categories.
We also note that Theorem~\bref{thm:3}  is the first example of a braided tensor equivalence between a braided tensor category obtained in a logarithmic conformal field theory and the representation category of a quantum group.

For factorisable Hopf algebras one can obtain an $SL(2,\oZ)$-action on the centre of the Hopf algebra~\cite{[LM],[Lyu]}. 
We observe in Appendix~\bref{app:SL2Z} that -- using the same expressions (despite working with a quasi-Hopf algebra) -- our R-matrix defines an $SL(2,\oZ)$-action on the centre of $\iUresSL2$, equivalent to the one given in \cite{[FGST]}, which agrees with $SL(2,\oZ)$-action via modular transformations on the space of torus conformal blocks of the $\mathcal{W}_2$-triplet algebra. 
\medskip

Let us give a brief outline of the proof of Theorem~\bref{thm:3}. The proof proceeds in three steps. First, we define an auxiliary algebra $\Salg$ in $\svect$ which has half the dimension of $\iUresSL2$. The algebra $\Salg$ is equipped with a non-coassociative coproduct, resulting in a tensor product functor on $\rep\Salg$ without -- so far -- an associator. We then establish the existence of multiplicative equivalences
\begin{equation}
	\catSF
	 ~\xrightarrow{~~~\funD~~~}~ 
	\rep\Salg
	 ~\xrightarrow{~~~\funSQ~~~}~ 
	\rep\iUresSL2 \ .
\end{equation}
Recall that
a functor $F$ is called \textit{multiplicative}\footnote{
	The definition is taken from \cite{Majid:1992}
	(see also \cite[Sect.\,9.4.1]{Majid-book}).
If the target of the functor is the category of vector spaces, the name {\em quasi-fibre functor} is more common.
}
if it is equipped with isomorphisms $\phi_{U,V} : F(U\tensor V) \to F(U)\tensor F(V)$, natural in $U,V$, and an isomorphism $\phi_{\one}: F(\one) \to \one$, on which no coherence conditions are imposed. A multiplicative functor is {\em monoidal} if $\phi_{U,V}$ and $\phi_{\one}$ satisfy  the hexagon and unit conditions.
On general grounds, it follows that there exists a
quasi-bialgebra structure on $\iUresSL2$, such that the composite functor $\funSQ \circ \funD$ is monoidal (see e.g.\ \cite[Sect.\,9.4]{Majid-book}). 

In step 2, we compute $\as$ by first finding a coassociator $\Sas$ on $\Salg$, turning it into a quasi-Hopf algebra in $\svect$, such that $\funD$ is monoidal. Then we transport $\Sas$ along $\funSQ$ to $\rep\iUresSL2$. 
The reason to pass via the intermediate category $\rep\Salg$ was that both $\rep\Salg$ and $\catSF$ are defined over the underlying category $\svect$, while $\rep\iUresSL2$ is defined over $\vect$. Passing via $\rep\Salg$ avoids excessive (and confusing) mixing of the tensor flips in $\svect$ and $\vect$.

Step 3 consists of transporting the braiding from $\catSF$ to $\rep\iUresSL2$ and reading off the R-matrix which gives rise to this braiding. Since $\catSF$ is a ribbon category, we can also transport the ribbon twist to $\rep\iUresSL2$ and we compute the corresponding ribbon element.

In presenting these steps in the body of the paper, we have opted for collecting the quasi-Hopf structure of $\iUresSL2$ and of $\Salg$ in one section each, rather than postponing the definition of the coassociator to a later section (which would be the chronological account). We believe this improves the readability of the paper.

\medskip

Starting from Theorem \bref{thm:3}, there are a number of further directions which are worth pursuing. 

Firstly, it should be relatively straightforward to generalise the construction of this paper to several pairs of symplectic fermions by taking appropriate products of the categories and algebras involved, and we expect to make contact with the results in \cite{Lentner:2014} on quantum groups at roots of unity of small order. 

Secondly, and more difficult, from the relation to $\rep\Walg_p$ one may expect the existence of a modified coproduct on $\UresSL2$ for $\q = e^{\rmi\pi/p}$ (for the same algebra structure) such that the tensor product becomes commutative and such that the structure of a quasi-triangular quasi-Hopf algebra exists. Again, a generalisation of $\Salg$ may serve as a helpful intermediate step. 
	Our construction
suggests that such an $\Salg$ should live in $\oZ_p$-graded vector spaces with an appropriate braiding. $\UresSL2$ is of dimension $2p^3$ and we would expect $\Salg$ to be of dimension $2p^2$. Of course, it would then be highly desirable to establish a braided monoidal equivalence with $\rep\Walg_p$, but associator and braiding in the latter category are not yet sufficiently explicitly understood.

Thirdly, and less directly linked to the present paper, there are several constructions related to conformal field theory which have been formulated for Hopf-algebras, and where one could take the example from Theorem \bref{thm:3} as a starting point to look for a generalisation to quasi-Hopf algebras. Of particular interest to us are the construction of \cite{Fuchs:2013lda}, which  provides mapping class group invariants that can serve as bulk correlation functions in logarithmic conformal field theory, and \cite{Blanchet:2014}, where a three-dimensional topological field theory is constructed starting from quantum $sl(2)$ at a root of unity (but different from the $\UresSL2$ used here).

\vspace{1mm}

The rest of the paper is organised as follows. In Section \bref{sec:QG} we look at $\iUresSL2$ in more detail and prove Theorem~\bref{thm:1}.
	In Section~\bref{sec:cat-C}, we review the braided category $\catSF$ of the symplectic fermions CFT.
	In Section \bref{sec:S}, we introduce  the quasi-Hopf algebra $\Salg$ in $\svect$. Sections \bref{sec:RepS-RepQ} and~\bref{sec:SF-RepS} detail the multiplicative functors $\funSQ$ and $\funD$.
	In Section~\bref{sec:proof}, we prove Theorems \bref{thm:2} and~\bref{thm:3}.

The proofs of 
Theorems~\bref{thm:1}, \bref{thm:2} and \bref{thm:3} 
rely heavily on computer algebra, specifically on Mathematica
	implementations of the quasi-Hopf algebras. We indicate below which steps of our proofs
	were done by computer algebra.

\bigskip

\noindent {\bf Acknowledgements}: We thank
Alexei Davydov,
	J\"urgen Fuchs,
Simon Lentner,
Alexei Semikhatov,
Volker Schomerus,
	Christoph Schweigert,
and Yorck Sommerh\"auser
	for helpful discussions
	and comments on a draft version of this paper.
A.M.G.\ was supported by Humboldt fellowship and RFBR-grant 13-01-00386.
A.M.G.\ wishes also to thank the IPhT in Saclay and Max-Planck Institute in Bonn
	for hospitality during the work on this project.
The authors are also grateful to the
	organisers 
of the program 
``Modern trends in topological quantum field theory'' at the Schr\"odinger Institute in Vienna in February and March 2014, where part of this work was undertaken.

\medskip

\noindent {\bf Notations}:
We use `$\cdot$' for the multiplication, e.g., $a\cdot b$ and the `$.$' is for the action, i.e., $a.u$ means that $a$ acts on $u$. We also write $a. u \tensor v$ for $(a.u) \tensor v$, and similarly for the right component,  in contrast to the action of $a$ on $u\tensor v$ denoted by $a.(u \tensor v)$.
Finally, to help the reader to navigate through this long paper, we provide a partial list of notations:
\begin{itemize}
\item $\catSF$ --- the braided tensor category obtained from the field theory of symplectic fermions, 
see \eqref{equiv-SF} and \eqref{eq:catSF-dec},
\item $\repQ$ --- the 
category of finite-dimensional representations of $\Q= \iUresSL2$, see Section~\bref{sec:repQ},
\item $\repS$ ---  the category of finite-dimensional super-vector space representations of $\Salg$ defined in Section~\bref{dfn:algS},
\item $\as$ --- the coassociator for $\Q$ given in \eqref{eq:as-intro} and \eqref{eq:Phi-via-f+-},
\item $\Sas$ ---  the coassociator for $\Salg$ in super-vector spaces, see \eqref{eq:Sas},
\item $\funSQ$ ---     the monoidal functor  from $\repS$ to $\repQ$ defined in Section~\bref{sec:funSQ-def},
\item $\isoG_{U,V}$ --- the family of isomorphisms~\eqref{isoG} of the functor $\funSQ$,
\item $\funCS$ ---     the monoidal functor  from $\catSF$ to $\repS$ defined in Section~\bref{sec:funD},
\item $\isoD_{U,V}$ --- the family of isomorphisms~\eqref{eq:isoD00}-\eqref{eq:isoD11} of the functor $\funCS$.
\item $\B$ --- a projective $\Salg$-module defined in~\eqref{eq:B}.
\end{itemize}

\section{Details on the restricted quantum group at $\q=\rmi$}\label{sec:QG}

\subsection{The restricted quantum group}
We will only consider the case $\q=\rmi$ (that is, $p=2$) for the restricted quantum group $\UresSL2$. We abbreviate
\begin{equation}
	\Q = \iUresSL2 \ .
\end{equation}
The defining relations specialise to
\begin{gather}\label{Ui-relations}
  \K\E\K^{-1}=-\E\ ,\quad
  \K\F\K^{-1}=-\F\ ,\quad
  [\E,\F]=\ffrac{\K}{\rmi}\ffrac{\one-\K^2}{2}\ ,\\ \nonumber
    \E^{2}=\F^{2}=0\ ,\quad \K^{4}=\one\ ,
\end{gather}
and the Hopf algebra structure is the same as described in the Introduction, i.e., as for generic~$\q$.
The algebra has dimension  $\mathrm{dim}\,\Q = 16$.

It turns out that the algebra $\Q$ is isomorphic to a direct sum of a Gra{\ss}mann and a Clifford algebra multiplied by a cyclic group of order two.
Indeed, introduce  the elements
\begin{equation}\label{eq:ferm-EF}
\ff^+=-\rmi\F\ ,\qquad \ff^-=\E\K^{-1}\ .
\end{equation}
They satisfy fermionic-type relations
\begin{equation}
\{\ff^+,\ff^-\}=\idem_1\ ,
\end{equation}
where we introduce central idempotents
\begin{equation}\label{eq:idem}
\idem_0 = \half(\one+\K^2)\ ,\qquad
\idem_1 = \half(\one-\K^2)\ .
\end{equation}
These central idempotents $\idem_i$ correspond to a decomposition of $\Q$ into ideals $\Q_i$:
\begin{equation}\label{Q-decomp}
\Q = \idem_0\Q \oplus \idem_1\Q = \Q_0\oplus \Q_1\ .
\end{equation}
We then see that the ideal $\Q_0$ is generated by $\idem_0\ff^{\pm}$ satisfying the Gra{\ss}mann algebra relations and $\idem_0\K$ is of order two. This is a non-semisimple subalgebra in $\Q$. Conversely, the ideal $\Q_1$ is a semisimple subalgebra and is isomorphic to a direct sum of two Clifford algebras.
We note also the comultiplication formulas in the new notations
\begin{equation}\label{eq:Delta-fpm}
\Delta(\ff^{\pm}) = \ff^{\pm}\tensor\one + \K^{-1}\tensor\ff^{\pm}\ .
\end{equation}
The coproduct of $\idem_i$ is given by
\begin{equation}\label{eq:cop-idem}
\Delta(\idem_0) = \idem_0\tensor\idem_0 + \idem_1\tensor\idem_1\ ,\qquad
\Delta(\idem_1) = \idem_0\tensor\idem_1 + \idem_1\tensor\idem_0\ .
\end{equation}

\subsection{The category $\repQ$}\label{sec:repQ}
The $\oC$-linear category $\repQ$ of finite-dimensional representations of $\Q$ decomposes as
\begin{equation}
\repQ = \repQ_0 \oplus \repQ_1 \ ,
\end{equation}
following the algebra decomposition~\eqref{Q-decomp}. Using the coproduct formulas~\eqref{eq:cop-idem}, we see that the tensor product $\tensor$ functor respects the $\oZ_2$-grading in $\repQ$, i.e., $U\tensor V\in \repQ_0$ for $(U,V)\in\repQ_i\times\repQ_i$ and $U\tensor V\in \repQ_1$ for $(U,V)\in\repQ_i\times\repQ_{i+1}$, for $i\in\oZ_2$.

We begin by recalling results on simple $\Q$-modules and then describe briefly their projective covers~\cite{Feigin:2005xs}.
There are four simple modules $\XX^{\pm}_{s}$, with $s=1,2$, and they are highest-weight $\Q$-modules. The $\Q$ action is defined as follows:
the modules $\XX^{\pm}_{1}$ are one-dimensional of weights $\pm1$ with respect to $\K$ and with the zero action by $\E$ and $\F$;
the modules $\XX^{\pm}_{2}$ are two-dimensional of the highest weights $\pm\rmi$, i.e., there exists a basis $\{v^{\pm}_0,v^{\pm}_1\}$ in $\XX^{\pm}_{2}$ and the action
\begin{align*}
&\K.v^{\pm}_0 = \pm\rmi v^{\pm}_0\ , 
&&\E .v^{\pm}_0 = 0\ , 
&&\F.v^{\pm}_0 =  v^{\pm}_1\ ,
\\
&\K .v^{\pm}_1 = \mp\rmi v^{\pm}_1\ ,
&&\E. v^{\pm}_{1} = \pm v^{\pm}_0\ ,
&&\F .v^{\pm}_{1} = 0 \ .
\end{align*}
We note that the simple modules $\XX^{\pm}_{2}$ are projective and  the $\XX^{\pm}_{1}$ have indecomposable but reducible projective covers $\PP^{\pm}_{1}$.
	The composition series of $\PP^{\pm}_{1}$ contains two copies each of $\XX^{+}_{1}$ and $\XX^{-}_{1}$, more details can be found in~\cite{Feigin:2005xs}.
$\PP^{\pm}_{1}$ are projective objects in $\repQ_0$ while the  $\XX^{\pm}_{2}$ are in $\repQ_1$.

\newcommand{\om}{\omega}
\newcommand{\eps}{\varepsilon}
\subsection{$\Q$ is not quasi-triangular} It is known~\cite{KS} that the Hopf-algebra $\UresSL2$ at integer $p>2$ is not quasi-triangular because
there are examples of finite-dimensional $\UresSL2$-modules $U$, $V$ such that $U\tensor V \not\cong V \tensor U$.
However, in the special case of $p=2$ (i.e.\ $\q=\rmi$) the tensor product satisfies
$U\tensor V \cong V \tensor U$ for all finite-dimensional $\UresSL2$-modules.
Nonetheless, we show below that $\UresSL2$ is not quasi-triangular even at $p=2$.

Instead of Theorem~\bref{thm:1}, we  prove a more general result. 
The third group cohomology $H^3(\oZ_2,\oC^*)$ is isomorphic to $\oZ_2$. Thus,
up to coboundaries there are two normalised $3$-cocycles for the group $\oZ_2$ (written additively) with values in $\oC^*$ (written multiplicatively), namely $\phi_{+1}$ and $\phi_{-1}$ with
\begin{equation*}
\phi_{\eps}(1,1,1) = \eps, \qquad \text{with} \quad \eps = \pm1 \ ,
\end{equation*}
and $\phi_\eps(i,j,k)=1$ if any of $i,j,k$ is $0$.
Each such $3$-cocycle gives an example of a quasi-bialgebra structure on $\Q$:
\begin{equation}
		\Phi_{\eps} = \sum_{i,j,k\in\oZ_2}\phi_{\eps}(i,j,k) \, \idem_i\tensor\idem_j\tensor\idem_k = \one\tensor\one\tensor\one + (\eps-1)\,\idem_1\tensor\idem_1\tensor\idem_1\ .
\end{equation}

The following theorem implies Theorem \bref{thm:1}:
\begin{Thm}
	The quasi-bialgebra $(\Q,\Phi_{\eps})$ is not quasi-triangular.
\end{Thm}
\begin{proof}
Using computer algebra (we used Mathematica) one shows that the conditions in Definition \bref{def:quasi-triang_for_quasi-hopf} have no solution. Since the conditions are non-linear, let us nonetheless give our procedure in some detail.

The R-matrix $R$ is an element of $\Q\tensor\Q$, a 256-dimensional vector space. The linear conditions
\begin{gather*}
R\Delta(\K) = \Delta^{\mathrm{op}}(\K)R ~~,\quad
R\Delta(\E) = \Delta^{\mathrm{op}}(\E)R ~~,\quad
R\Delta(\F) = \Delta^{\mathrm{op}}(\F)R \\
 (\epsilon \tensor \id)(R) = \one = (\id \tensor \epsilon)(R) \ .
\end{gather*}
have a 99-dimensional affine subspace of $\Q\tensor\Q$ as set of solutions.
We insert this candidate-R-matrix with 99 free parameters into the two hexagon conditions \eqref{eq:R-mat-hex12}. The hexagon conditions are quadratic and depend on $\Phi_{\eps}$. While the details of the calculation differ slightly between $\eps=+1$ and $\eps=-1$, the procedure is the same in both cases:

Looking at the resulting equations in $\Q^{\otimes 3}$ component by component reveals many {\em linear} conditions on the 99 parameters. After imposing all linear conditions and quadratic conditions of the form $\text{(parameter)}^2=0$ (iteratively, as more appear in the process), one finds a quadratic equation of the form $(x-a)(y-b)=0$, where $x,y$ are two of the remaining parameters and $a,b \in\oC$ are constants. One then checks that both possibilities, $x=a$ or $y=b$, lead to contradicting conditions in the remaining equations.
\end{proof}

If one complements	the quasi-bialgebra $(\Q,\Phi_{\eps})$
by the evaluation and coevaluation elements
(see Proposition~\bref{prop:antipode})
\begin{equation}
	\Salpha_{\eps} = \one
\,, \qquad \Sbeta_{\eps} = \idem_0 + \eps \idem_1 \,,
\end{equation}
one easily checks that one obtains a quasi-Hopf algebra.

\section{The braided tensor category $\catSF$
	from
symplectic fermions}\label{sec:cat-C}

We consider the chiral conformal field theory of one pair of symplectic fermions. Symplectic fermions first appeared in \cite{Kausch:1995py} and are by now the best investigated logarithmic conformal field theory, see e.g.\ \cite{Gaberdiel:1996np,Fuchs:2003yu,Abe:2005,Feigin:2005xs,
Abe:2011ab,Runkel:2012cf}.

In this section we summarise the construction of \cite{Davydov:2012xg,Runkel:2012cf} which uses the monodromy properties and asymptotic behaviour of symplectic fermions conformal blocks to produce a braided tensor category $\catSF$.

\subsection{Mode algebra and representations}

The mode algebra of a single pair of symplectic fermions comes in two versions: twisted and untwisted. Both are Lie superalgebras with odd generators $\chi^\pm_m$, a central even generator $K$ and anticommutation relations
\begin{equation}
	\{ \chi^+_m,\chi^-_n \} = m \, \delta_{m+n,0} \,K \ .
\end{equation}
For the untwisted mode algebra $\hat{\mathfrak{h}}$ we have $m,n \in \oZ$, while for the twisted mode algebra $\hat{\mathfrak{h}}_\text{tw}$, $m,n \in \oZ + \frac12$.

The representations of $\hat{\mathfrak{h}}$ we are interested in are in super-vector spaces, such that the action of $\chi^\pm_m$ is by odd maps, while $K$ acts as $1$. Furthermore, the representations
\begin{itemize}
\item[-] are bounded below in the sense that every vector is annihilated by each word in the generators of sufficiently positive total mode number,
\item[-] have finite-dimensional highest weight space (the subspace annihilated by all $\chi^\pm_m$ with $m>0$).
\end{itemize}
Denote this category by
$\rep_{\flat,1}^\text{fd}(\hat{\mathfrak{h}})$. The category $\rep_{\flat,1}^\text{fd}(\hat{\mathfrak{h}}_\text{tw})$ is constructed analogously.

Write $\chi^{\pm} := \chi^{\pm}_0$ for the zero modes in $\hat{\mathfrak{h}}$. They generate a four-dimensional Gra\ss mann algebra $\algGr$,
\begin{equation}
\{\chi^+,\chi^-\}=0 \ .
\end{equation}
We will work with the normalisation used in \cite{Davydov:2012xg} instead of that in \cite{Runkel:2012cf} as it involves fewer factors of $\pi \rmi$. That is, we use the generators
\begin{equation}
	a_1 := \chi^+
~~,\quad
	a_2 := \pi \rmi \, \chi^-
~~,\quad
	\{a_1,a_2\} = 0 \ .
\end{equation}
The Gra\ss mann algebra $\algGr$ becomes a Hopf algebra in $\svect$ by giving $a_{1}, a_2$ odd parity and choosing $\Delta(a_j) = a_j \otimes \one + \one \otimes a_j$, $S(a_j) = - a_j$, for $j=1,2$.

It is shown in \cite[Thms.\,2.4\,\&\,2.8]{Runkel:2012cf} that the functor associating to a representation in $\rep_{\flat,1}^\text{fd}(\hat{\mathfrak{h}})$ or $\rep_{\flat,1}^\text{fd}(\hat{\mathfrak{h}}_\text{tw})$ its highest weight space, provides equivalences of $\oC$-linear categories as follows:
\begin{equation}\label{equiv-SF}
	\rep_{\flat,1}^\text{fd}(\hat{\mathfrak{h}})
	\xrightarrow{~~\sim~~}
	\repsv \algGr =: \catSF_0
\quad , \qquad
	\rep_{\flat,1}^\text{fd}(\hat{\mathfrak{h}}_\text{tw})
	\xrightarrow{~~\sim~~}
	\svect =: \catSF_1 \ .
\end{equation}
Here, $\repsv \algGr$ denotes the category of finite-dimensional representations of $\algGr$ in super-vector spaces, respecting the $\oZ_2$-grading of $\algGr$. We write
\begin{equation}\label{eq:catSF-dec}
\catSF := \catSF_0 \oplus  \catSF_1
\end{equation}
for the category whose objects and morphism spaces are direct sums of those in $\catSF_0$ and $\catSF_1$.

Next, we will use the chiral conformal field theory of symplectic fermions to endow $\catSF$ with a braided monoidal structure.

\subsection{Tensor product}

The tensor product $X * Y$ of two objects $X,Y \in \catSF$ is defined as a representing object for the functor which assigns to an object $Z \in \catSF$ the space of vertex operators from $X$ and $Y$ to $Z$ -- or rather between their preimages in $\rep_{\flat,1}^\text{fd}(\hat{\mathfrak{h}}_\text{(tw)})$
under the equivalence
in~\eqref{equiv-SF}. For details we refer to \cite[Sect.\,3]{Runkel:2012cf}. The result is \cite[Thm.\,3.13]{Runkel:2012cf}:
\begin{equation}\label{eq:*-tensor}
X \ast Y ~=~
\left\{\rule{0pt}{2.8em}\right.
\hspace{-.5em}\raisebox{.7em}{
\begin{tabular}{ccll}
   $X$ & $Y$ & $X \ast Y$ &
\\
$\catSF_0$ & $\catSF_0$ & $X \otimes_{\repsv\algGr} Y$ & $\in~\catSF_0$
\\
 $\catSF_0$ & $\catSF_1$ & $F(X) \svtensor Y$ & $\in~\catSF_1$
\\
 $\catSF_1$ & $\catSF_0$ & $X \svtensor F(Y)$ & $\in~\catSF_1$
\\
$\catSF_1$ & $\catSF_1$ & $\algGr \svtensor X \svtensor Y$ & $\in~\catSF_0$
\end{tabular}}
\end{equation}
Here, $F : \repsv\algGr \to \svect$ is the forgetful functor, and $X \otimes_{\repsv\algGr} Y$ stands for $X \svtensor Y$ with $\algGr$-action via the coproduct. On morphisms we simply have $f * g = f \tensor g$ in all cases but the last, where $f *g = \id_G \tensor f \tensor g$.

We remark that in $\svect$, the action of $\algGr$ on $X \svtensor Y$ involves the symmetric braiding
\begin{equation}\label{sflip}
\sflip_{X,Y} : X \svtensor Y \longrightarrow Y \svtensor X
\quad , \quad
\sflip_{X,Y}(x \tensor y) = (-1)^{|x||y|} y \tensor x \ ,
\end{equation}
of $\svect$ (where $x,y$ are homogeneous of $\oZ_2$-degree $|x|$, $|y|$). Explicitly, the action map $\rho^{X \otimes Y} : \algGr \svtensor X \svtensor Y \to X \svtensor Y$ is given by
\begin{equation}\label{eq:action-on-tensor-svect}
\rho^{X \otimes Y} = (\rho^X \tensor \rho^Y) \circ (\id_{\algGr} \tensor \sflip_{\algGr,X} \tensor \id_Y) \circ (\Delta \tensor \id_X \tensor \id_Y) \ ,
\end{equation}
or, on elements,
\begin{equation}
	g.(x\otimes y) = \sum_{(g)} (-1)^{|g''||x|} (g'.x) \otimes (g''.y)
\quad \text{where} \quad \Delta(g) = \sum_{(g)} g' \otimes g'' \ .
\end{equation}

\subsection{Associator}\label{sec:assoc-SF}

The associator of $\catSF$ is computed in \cite{Runkel:2012cf} via the usual procedure of comparing asymptotic behaviour of four-point conformal blocks. The four-point blocks in turn are defined via appropriate compositions of two vertex operators. Here we just list the result of this calculation.

Both, the associator and the braiding are expressed in terms of a constant $\beta \in \oC$ 
	satisfying \eqref{eq:beta-param}
and a copairing $C$ on $\algGr$
given by \cite[Eqn.\,(5.19)]{Davydov:2012xg}\footnote{
	To obtain the braided monoidal category of symplectic fermions as computed in \cite{Runkel:2012cf}, one has to set $\beta = e^{- \pi \rmi/4}$.
	The relation to the copairing $\Omega=\chi^- \tensor \chi^+ - \chi^+ \tensor \chi^-$ used in \cite{Runkel:2012cf} is $C = \pi \rmi \Omega$.}
\begin{equation}
	C := a_2 \tensor a_1 - a_1 \tensor a_2
     \quad \in~ \algGr \svtensor \algGr \ .
\end{equation}

The associator is a natural family of isomorphisms
\begin{equation*}
\alpha^{\catSF}_{X,Y,Z} :\quad X*(Y*Z) \to (X*Y)*Z \ .
 \end{equation*}
 Its form depends on whether $X,Y,Z$ are chosen from $\catSF_0$ or $\catSF_1$. There are eight possibilities, which we now list and then explain the notation (see \cite[Thm.\,6.2]{Runkel:2012cf} and \cite[Sect.\,5.2\,\&\,Thm.\,2.5]{Davydov:2012xg}):
\begin{align*}
X\,Y\,Z \qquad & X \ast (Y \ast Z) && (X \ast Y) \ast Z && \alpha^{\catSF}_{X,Y,Z} : X \ast (Y \ast Z) \to (X \ast Y) \ast Z
\nonumber\\[.3em]
0\,~0\,~0\, \qquad &
\underline{X} \tensor \underline{Y} \tensor \underline{Z} &&
\underline{X} \tensor \underline{Y} \tensor \underline{Z} &&
\id_{X \otimes Y \otimes Z}
\nonumber\\[.3em]
0\,~0\,~1\, \qquad &
X \tensor Y \tensor Z &&
X \tensor Y \tensor Z &&
\id_{X \otimes Y \otimes Z}
\nonumber\\[.3em]
0\,~1\,~0\, \qquad &
X \tensor Y \tensor Z &&
X \tensor Y \tensor Z &&
\exp\!\big(C^{(13)} \big)
\nonumber\\[.3em]
1\,~0\,~0\, \qquad &
X \tensor Y \tensor Z &&
X \tensor Y \tensor Z &&
 \id_{X \otimes Y \otimes Z}
\nonumber\\[.3em]
0\,~1\,~1\, \qquad &
\underline{X} \tensor \underline{\algGr}\tensor Y \tensor Z &&
\underline{\algGr} \tensor X\tensor Y \tensor Z &&
\Big[\big\{ \id_{\algGr} \otimes (\rho^X \circ (S \otimes \id_X))  \big\}
\nonumber\\[-.1em]
&&&&&\qquad \circ \big\{ 
	\Delta
\otimes \id_X \big\} \circ \sflip_{X,\algGr} \Big] \otimes \id_{Y \otimes Z}
\nonumber\\[.3em]
1\,~0\,~1\, \qquad &
\underline{\algGr} \tensor X\tensor Y \tensor Z &&
\underline{\algGr} \tensor X\tensor Y \tensor Z &&
\exp\!\big( C^{(13)} \big)
\nonumber\\[.3em]
1\,~1\,~0\, \qquad &
\underline{\algGr} \tensor X\tensor Y \tensor Z &&
\underline{\algGr} \tensor X \tensor Y \tensor \underline{Z} &&
\big\{ \id_{\algGr \otimes X \otimes Y} \otimes \rho^Z \big\}
  \circ
  \big\{ \id_{\algGr} \otimes \sflip_{\algGr,X \otimes Y} \otimes \id_Z  \big\}
\nonumber\\[-.1em]
&&&&&\qquad
  \circ
  \big\{ \Delta \otimes \id_{X \otimes Y \otimes Z} \big\}
\nonumber\\[.3em]
1\,~1\,~1\, \qquad &
X \tensor \algGr \tensor Y \tensor Z &&
\algGr \tensor X \tensor Y \tensor Z &&
\big\{ \phi \otimes \id_{X \otimes Y \otimes Z} \big\} \circ \big\{\sflip_{X,\algGr} \otimes \id_{Y \otimes Z}\big\}
\end{align*}
If the sector `1' occurs an even number of times, the triple tensor product lies in $\catSF_0$ and hence carries an action of $\algGr$. The underlines indicate on which  tensor factors $\algGr$ acts (action on multiple factors is always via the coproduct). The action of $C^{(13)}$ is given by letting the first tensor factor of $C$ act on $X$ and the second on $Z$ (using the braiding of $\svect$ to move elements past each other). Explicitly,
\begin{equation}
	C^{(13)}(x \tensor y \tensor z) =
(-1)^{|x|+|y|} \big( a_2.x \tensor y \tensor a_1.z - a_1.x \tensor y \tensor a_2.z \big) \ ,
\end{equation}
for homogeneous $x,y,z$. The linear map $\phi : \algGr \to \algGr$ is given by
\begin{equation}\label{eq:SF-phi-def}
	\phi(1) = \beta^2 \, a_1 a_2
~~,\quad
	\phi(a_1) = \beta^2 \, a_1
~~,\quad
	\phi(a_2) = \beta^2 \, a_2
~~,\quad
	\phi(a_1 a_2) = - \beta^2 \, 1
\ .
\end{equation}

\subsection{Braiding}\label{sec:SF-braiding} The braiding on $\catSF$ is obtained from the monodromy properties of symplectic fermion three-point blocks.\footnote{
	Actually, this initially only produces a braiding in the sense of crossed $\oZ_2$-categories \cite[Sect.\,4]{Runkel:2012cf}. However, with the help of the parity involution on $\svect$, this can be turned into the `proper' braiding as stated in Section~\bref{sec:SF-braiding}, see  \cite[Sect.\,6.3]{Runkel:2012cf}.}
The resulting family of natural isomorphism $c_{X,Y}$ is  (see \cite[Thm.\,6.4]{Runkel:2012cf} and \cite[Sect.\,5.2\,\&\,Thm.\,2.8]{Davydov:2012xg}):
\begin{equation*}
\raisebox{.7em}{
\begin{tabular}{ccl}
 $X$ & $Y$ & $c_{X,Y} ~:~ X \ast Y \to Y \ast X$
\\[.3em]
 $0$ & $0$ & $ \sflip_{X,Y} \circ \exp(- C) $
\\[.3em]
 $0$ & $1$ & $\sflip_{X,Y} \circ \big\{ \exp\!\big( \tfrac12\hat C\big) \tensor \id_Y \big\}$
\\[.3em]
 $1$ & $0$ & $\sflip_{X,Y}  \circ \big\{ \id_X \tensor  \exp\!\big( \tfrac12\hat C\big) \big\} \circ \{ \id_X \otimes \omega_Y \}$
\\[.3em]
 $1$ & $1$ & $\beta \cdot
  \big(\id_{\algGr} \otimes \sflip_{X,Y}\big)\circ \big\{ \exp\!\big(- \tfrac12\hat C\big) \otimes \id_X \tensor \omega_Y \big\}$
\end{tabular}}
\end{equation*}
Here, $\hat C$ is the multiplication of $\algGr$ applied to $C$, i.e.\ $\hat C = -2 a_1 a_2$. For $X \in \svect$,
\begin{equation}
	\omega_X : X \xrightarrow{~\sim~} X
\quad,\quad
	\omega_X(x) = (-1)^{|x|} \, x \ ,
\end{equation}
denotes the parity involution on $X$. The family $X \mapsto \omega_X$ is a natural monoidal isomorphism of the identity functor on $\svect$.

\subsection{Ribbon twist}\label{eq:SF-twist}

The braided monoidal structure on $\catSF$ can be enhanced to a ribbon structure, see \cite[Sect.\,4]{Davydov:2012xg}. Here we will only make use of the ribbon twist isomorphisms $\theta_X$, which are given by \cite[Rem.\,6.5]{Runkel:2012cf}:
\begin{equation*}
\raisebox{.7em}{
\begin{tabular}{ccl}
 $X$ & $\theta_X : X \to X$
\\[.3em]
 $0$ & $ \exp(-\hat C) $
\\[.3em]
 $1$ & $\beta^{-1} \cdot \omega_X$
\end{tabular}}
\end{equation*}
The twist isomorphisms satisfy
\begin{equation}
	\theta_{X * Y} = (\theta_X \tensor \theta_Y) \circ c_{Y,X} \circ c_{X,Y} \ .
\end{equation}

For symplectic fermions, the ribbon twist	acts
by $\exp(-2 \pi \rmi L_0)$ on the parity even subspace of a given representation. The even subspaces of the four irreducibles in $\rep_{\flat,1}^\text{fd}(\hat{\mathfrak{h}}) \oplus \rep_{\flat,1}^\text{fd}(\hat{\mathfrak{h}}_\text{tw})$ have lowest $L_0$ eigenvalues $h$ as follows:
\begin{center}
\begin{tabular}{cc}
untwisted & \qquad\quad twisted \\
$h=0$,\quad $h=1$ & \quad\qquad $h = -\frac18$,\quad $h=\frac38$
\end{tabular}
\end{center}
This agrees with the expression in $\catSF$ given above for  $\beta = e^{- \pi \rmi/4}$, as it should.

\bigskip

The subcategory $\catSF_0$ of $\catSF$ is closed under the tensor product $*$. When restricted to $\catSF_0$, the associator in Section \bref{sec:assoc-SF} becomes trivial.
The braiding $\brC_{U,V}$ restricted to $\catSF_0$ can be described by an R-matrix,
\begin{equation}\label{eq:C00-br}
\brC_{U,V} = \sflip_{U,V}\circ \Rsf \ ,
\end{equation}
with $\Rsf = \exp(-C) =1 - C + \frac12 C^2$ (since $C^3=0$). When evaluating $\Rsf$ on elements, one has to keep track of the parity signs which arise in $C^2$ and when acting on $u \tensor v$,
\begin{equation}
\Rsf(u\tensor v) = u\tensor v -(-1)^{|u|}a_2.u\tensor a_1.v + (-1)^{|u|}a_1.u\tensor a_2.v
-a_1a_2.u\tensor a_1a_2.v \ .
\end{equation}
The ribbon twist on $\catSF_0$ is given by $\theta_U(u) = u + 2a_1a_2.u$.

\section{A quasi-Hopf algebra in $\svect$}\label{sec:S}
Here we will introduce the quasi-Hopf algebra $\Salg$ in $\svect$, which in Section~\bref{sec:proof}
will serve as an intermediate step when transporting the monoidal  structure, braiding
and the ribbon twist from
$\catSF$ to $\repQ$.
   One can try to skip this step (and we did try), but the different braidings in $\vect$ and $\svect$ become cumbersome to deal with.

We will begin by collecting the various structures on $\Salg$ in the next definition and then will provide the proofs for the claims made in the process.
The definition of the coproduct and coassociator of $\Salg$ will seem opaque and ad hoc. They result from considering $\oC$-linear equivalences $\catSF \to \repS$ and $\repS \to \repQ$ and turning them into monoidal equivalences. We will elaborate on this in Remark~\bref{rem:origin-DeltaS} and Section~\bref{rem:Lambda-from-SF} below.

\begin{Dfn}\label{dfn:algS}
$\Salg$ is the eight-dimensional associative algebra over $\oC$ generated by $\xpm$ and $\LL$ with the defining relations
\begin{equation}
\xpm\LL=\LL\xpm ~~,\quad \{\xp,\xm\}=\half(\one-\LL)~~,\quad
(\xpm)^2=0~~,\quad \LL^2=\one \ .
\end{equation}
The $\oZ_2$-grading is such that $\xpm$ have odd degree and $\LL$ has even degree.
Define the central idempotents
\begin{equation}\label{eq:S-e0-e1-def}
\idem_0=\half(\one+\LL) \quad , \quad \idem_1=\half(\one-\LL) \ .
\end{equation}
They give a decomposition
\begin{equation}\label{eq:Salg-dec}
\Salg = \Salg_0 \oplus \Salg_1~~,
\qquad \text{where} \quad \Salg_0 := \idem_0\Salg ~~,~~\Salg_1 := \idem_1\Salg \ ,
\end{equation}
of $\Salg$ into a four-dimensional Gra\ss mann algebra $\Salg_0$ and a four-dimensional Clifford algebra $\Salg_1$.

The algebra $\Salg$ is equipped with a non-coassociative coproduct and a counit, namely with the algebra maps $\copS: \Salg\to\Salg\tensor\Salg$ and $\epsilon:\Salg\to\oC$ given by
\begin{align}
\copS(\xpm) &= \xpm\tensor\one + (\idem_0-\rmi\idem_1)\tensor\xpm \pm \rmi\idem_1\tensor\idem_1(\xp-\xm) ~~,
&
\epsilon(\xpm)&=0
\ ,\label{copS}\\
\nonumber
\copS(\LL) &=\LL\tensor\LL ~~,
&
\epsilon(\LL) &=1
\ .
\end{align}
We define the antipode $S : \Salg \to \Salg$ to be the algebra anti-automorphism determined by
\begin{equation}\label{eq:Salg-S}
S(\xpm) = -\xpm(\idem_0 - \rmi\idem_1),\qquad S(\LL)=\LL \ .
\end{equation}
Note that an anti-automorphism in $\svect$ satisfies
$S(ab) = (-1)^{|a||b|}S(b)S(a)$ for all homogeneous $a,b\in\Salg$.

The coassociator $\Sas \in \Salg^{\tensor 3}$ depends on a complex parameter $\beta$ satisfying $\beta^4=-1$, cf.\ \eqref{eq:beta-param}. It is given by
\begin{align}\label{eq:Sas}
\Sas = & ~ \idem_0\tensor\idem_0\tensor\idem_0 \\
&
+ \idem_1\tensor\idem_0\tensor\idem_0
+ \Sas^{010}\idem_0\tensor\idem_1\tensor\idem_0
+ \idem_0\tensor\idem_0\tensor\idem_1
\nonumber\\
& + \Sas^{110}\idem_1\tensor\idem_1\tensor\idem_0
  + \Sas^{101}\idem_1\tensor\idem_0\tensor\idem_1
  +\idem_0\tensor\idem_1\tensor\idem_1
\nonumber\\
& + \Sas^{111}\idem_1\tensor\idem_1\tensor\idem_1 ,
\nonumber
\end{align}
where
\begin{align*}
\Sas^{010} =~& \one\tensor\one\tensor\one
    + (1 + \rmi) \xm\tensor\one\tensor\xp - (1 - \rmi) \xp\tensor\one\tensor\xm -
  2 \xp\xm\tensor\one\tensor\xp\xm,
\\
\Sas^{110} =~& \one\otimes \one\otimes \one - \rmi \one\otimes (\xp + \xm)\otimes (\xp + \xm),\\
\Sas^{101} =~& \one\otimes \one\otimes \one + \rmi \one\otimes (\xp + \rmi\xm)\otimes (\xp + \rmi\xm) + (1 - \rmi)(\xp- \xm)\otimes \xp  \xm\otimes (\rmi\xp-\xm)\\
    &+ (1 + \rmi) \xm\otimes \xp\otimes \one
    - (1 - \rmi) \xp\otimes \xm\otimes \one
 + (1 + \rmi) \one\otimes \xp  \xm\otimes \one
     - 2 \xp  \xm\otimes \xp  \xm\otimes \one,\\
\Sas^{111} =~& \beta^2\big\{(\rmi-1) \one\otimes \one\otimes \one
        + (1 - \rmi) \one\otimes \one\otimes \xp\xm + (1 + 2 \rmi) \one\otimes \xm\otimes \xm 
        + \one\otimes \xm\otimes \xp 
\\ &
        + \one\otimes \xp\otimes \xm
+   \one\otimes \xp\otimes \xp + (1 - \rmi) \one\otimes \xp\xm\otimes \one - (2 - 2 \rmi) \one\otimes \xp\xm\otimes \xp\xm 
\\ &
- (1 - \rmi) \xm\otimes \one\otimes \xm 
+ (1 + \rmi) \xm\otimes\one\otimes \xp - (1 - \rmi) \xm\otimes \xm\otimes \xp  \xm 
\\ &
+ (1 + \rmi) \xm\otimes \xp\otimes \xp  \xm 
+ (1 - \rmi) \xm\otimes \xp  \xm\otimes \xm
- (1 + \rmi) \xm\otimes \xp  \xm\otimes \xp 
\\ &
+ (1 - \rmi) \xp\otimes \xm\otimes \one - (1 - \rmi) \xp\otimes \xm\otimes \xp  \xm 
- (1 + \rmi) \xp\otimes \xp\otimes \one 
\\ &
+ (1 + \rmi) \xp\otimes \xp\otimes \xp  \xm - (1 - \rmi) \xp\otimes \xp\xm\otimes \xm
+ (1 + \rmi) \xp\otimes \xp  \xm\otimes \xp 
\\ &
+ 2 \xp  \xm\otimes \one\otimes \one - 2 \xp  \xm\otimes \one\otimes \xp  \xm - 2 \rmi \xp  \xm\otimes \xm\otimes \xm 
- 2 \rmi \xp\xm\otimes \xp\otimes \xp 
\\ &
- 2 \xp\xm\otimes \xp  \xm\otimes \one + 4 \xp  \xm\otimes \xp  \xm\otimes \xp  \xm
\big\} \ .
\end{align*}

We fix the evaluation element $\Salpha$ and coevaluation element $\Sbeta$ as
\begin{equation}\label{eq:Salg-alpha}
\Salpha = \idem_0 +\idem_1(\xp+\xm) ~~,\quad
\Sbeta = \idem_0 + \beta^2 \idem_1(\xp - \xm) \ .
\end{equation}
\end{Dfn}

\begin{Prop}\label{prop:S-quasiHopf}
	The data $(\Salg,\cdot,\one,\copS,\epsilon,
	\Sas,
S,\Salpha,\Sbeta)$ is a quasi-Hopf algebra in $\svect$.
\end{Prop}

 We have actually two different quasi-Hopf structures on $\Salg$ depending on $\beta^2 \in \{\pm\rmi\}$.
 As for $\catSF$, we will not indicate the $\beta$ dependence in the notation for $\Salg$.

The proof of Proposition~\bref{prop:S-quasiHopf} is given in the next two lemmas.

\begin{Lemma}\label{lem:DeltaS-alg-map}
	$\eps : \Salg \to \oC^{1|0}$ and $\copS:\Salg\to\Salg\svtensor\Salg$ are algebra homomorphisms; $S : \Salg \to \Salg$ is an algebra anti-homomorphism.
\end{Lemma}

\begin{proof}
For $\eps$ the statement is clear.  For $\copS$ a straightforward calculation in $\svect$ shows that
$\copS(\xpm)\copS(\xpm)=0$,
$\copS(\LL)\copS(\xpm)=\copS(\xpm)\copS(\LL)$, and
$\copS(\xp)\copS(\xm)+\copS(\xm)\copS(\xp)=\copS(\idem_1)$.
For $S$ one needs to check (note the signs in the last relation)
$S(\xpm)S(\xpm)=0$,
$S(\LL)S(\xpm)=S(\xpm)S(\LL)$, and
$-S(\xp)S(\xm)-S(\xm)S(\xp)=S(\idem_1)$, which is again straightforward.
\end{proof}

However, $\copS$ is not coassociative, since, e.g.,
\begin{align*}&\big((\copS\tensor\id)\circ \copS\big)(\xp+\xm)
- \big((\id\tensor\copS) \circ \copS\big)(\xp+\xm) \\&\qquad =
\tfrac12 (\one - \LL) \tensor (\one - \LL) \tensor (\xp+\xm) \neq 0 \ .
\end{align*}

\begin{Lemma}
Conditions \eqref{eq:eps-Delta}--\eqref{eq:3-cocycle} in Definition \bref{def:quasi-bialgebra} hold; conditions \eqref{eq:Salpha-1} and \eqref{eq:Salpha-2} in Definition \bref{def:quasi-Hopf_ass1} hold.
\end{Lemma}

\begin{proof}
The conditions have been checked by computer algebra. (The conditions not involving $\Sas$ can also be checked by hand, it is just $\Sas$ that is impractical to work with.)
\end{proof}

According to Definitions \bref{def:quasi-bialgebra} and \bref{def:quasi-Hopf_ass1}, the above two lemmas establish that $\Salg$ is a quasi-Hopf algebra.

\section{An equivalence from $\repS$ to $\repQ$}\label{sec:RepS-RepQ}

We denote by $\repS$ the category of finite-dimensional super-vector space representations of $\Salg$. According to the decomposition~\eqref{eq:Salg-dec}, we have the decomposition
\begin{equation}\label{eq:Z2-grading-repS}
\repS = \repS_0 \oplus \repS_1
\end{equation}
of the representation category.

By Lemma \bref{lem:DeltaS-alg-map}, the map $\copS$ defines a tensor product on $\repS$. Since the coproduct $\copS$ applied to $\idem_0,\idem_1$ takes the form  \eqref{eq:cop-idem}, the decomposition \eqref{eq:Z2-grading-repS} is compatible with the tensor product.

\subsection{The functor $\funSQ$}\label{sec:funSQ-def}
Define the $\oC$-linear functor $\funSQ: \repS \to \repQ$ as follows.
For a given $U\in\repS$, let $\funSQ(U)$ be the $\Q$-module with underlying vector space $U$, and with $\Q$-action given by, for $u \in U$,
\begin{align}\label{eq:funSQ}
\K. u &~:=~ \z . \omega_U(u)  = \omega_U(\z . u) ~~ , \quad \text{where} \quad \z=\idem_0+\rmi\idem_1 \ ,
\\
\nonumber
\ff^{\pm}. u &~:=~ \xpm. u \ .
\end{align}
Here, $\omega_U$ is the parity involution on the super-vector space $U$.
It is easy to see  that \eqref{eq:funSQ} indeed gives  a $\Q$-module structure on $U$. In particular, $\K.(\K. u)= \z.\omega_U\big(\z . \omega_U(u)\big) = \LL. u$, where we used $\z^2 = \LL$. So, $\K^2$ acts as $\LL$ and $\idem_i\in\Q$ ($i=0,1$) acts as $\idem_i\in\Salg$, making the notation consistent. On an $\Salg$-module map $f : U \to V$ we set $\funSQ(f) = f$.

\begin{Prop}\label{prop:G-equiv}
The functor $\funSQ$ is an equivalence of $\oC$-linear categories.
\end{Prop}
\begin{proof}
We will construct a functor $\fun:\repQ\to\repS$ which is inverse to $\funSQ$. We start by giving a $\oZ_2$-grading on $\Q$-modules.
Namely, inverting the relation in \eqref{eq:funSQ} we introduce the involution $\omega_V$ on an object $V \in \repQ$ as
\begin{equation}\label{Z2-repQ}
\omega_V(v) := (\idem_0-\rmi\idem_1)\K. v~~,\qquad v\in V \ ,
\end{equation}
where $\idem_0, \idem_1$ refer to \eqref{eq:idem}.
It is easy to check that $(\omega_V)^2$ is the identity.
The parity $|v|$ of an eigenvector $v$ of $\omega_V$ is determined by the eigenvalue as $\omega_V(v) = (-1)^{|v|} v$.
In particular, the generators $\ff^{\pm}$ of $\Q$ change the parity by one.\footnote{
It is important to remark that $\Q$ is {\em not} an algebra in $\svect$. Firstly, the identity $\one$ of $\Q$ is not homogeneous (and so in particular not parity even):  $\omega_{\Q}(\one) =  (\idem_0-\rmi\idem_1)\K \ne \pm(\idem_0+\idem_1)$. Secondly, the product does not respect parity (and so is not a morphism in $\svect$). For example, for $a=b=\one$ we have $\omega(ab) \ne \omega(a)\omega(b)$.}
We have thus embedded $\repQ$ into $\svect$.

Given $V\in\repQ$, the $\Salg$-module $\fun(V)$ has $V$ as underlying super-vector space with the $\oZ_2$-grading defined by the above involution $\omega_V$. The action of $\LL$ on $\fun(V)$ is given by the action of $\idem_0-\idem_1\in\Q$, while the action of $\xpm$ is given by the action of $\ff^{\pm}$. On morphisms $f$ in $\repQ$ we set $\fun(f) = f$.

It is an easy check that the compositions $\fun\circ\funSQ$ and $\funSQ\circ\fun$ are
	actually equal (and not just naturally isomorphic)
to the identity functors $\one_{\repS}$ and $\one_{\repQ}$, respectively. 
Thus, the two categories are even isomorphic.
\end{proof}

\subsection{$\funSQ$ as multiplicative functor}

To make $\funSQ$ multiplicative, we need to find a family of isomorphisms
\begin{equation}\label{isoG}
\isoG_{U,V}:\; \funSQ(U\tensor_{\repS} V) \to \funSQ(U)\Qtensor\funSQ(V) \ ,
\end{equation}
natural in $U$ and $V$. We claim that
\begin{equation}\label{eq:isoG}
\isoG_{U,V}(u \tensor v ) := u \tensor v + \idem_1. u \tensor (\xi-1)\idem_1. v \quad , \quad \text{where} ~~\xi = \xp + \xm \ ,
\end{equation}
does the job. Invertibility is easy to see since $(\isoG_{U,V})^2 = \id_{U \otimes V}$, which is an immediate consequence of $\xi^2 = \idem_1$. Naturality is also clear. It remains to show:

\begin{Lemma}
$\isoG_{U,V}$ is an intertwiner of $\Q$-modules.
\end{Lemma}

\begin{proof}
We need to show that for all $a \in \Q$, $u \in U$, $v \in V$ we have
\begin{equation}\label{eq:GammaUV-Q-intertw}
	\isoG_{U,V}(a.(u \otimes v)) = a.\isoG_{U,V}(u \otimes v) \ .
\end{equation}
It is enough to verify this on the generators $\K$, $\ff^\pm$. For example, for $a=\ff^+$, the two sides of the above identity are, with $\bar\z = \idem_0 - \rmi \idem_1$,
\begin{align}
	\isoG_{U,V}\big(\ff^+.(u \otimes v)\big)
	&= 	\isoG_{U,V}\big(\xp.(u \otimes v)\big)
\\\nonumber
	&= 	\isoG_{U,V}\big( \xp. u\tensor v
+ (-1)^{|u|}\bar\z. u\tensor\xp. v
\\\nonumber
& \hspace{3em}+ (-1)^{|u|} \rmi\idem_1. u\tensor\idem_1(\xp-\xm). v\big)
\ ,
\\\nonumber
	\ff^+.\isoG_{U,V}\big(u \otimes v\big)
&= (\ff^+\tensor \one + \K^{-1} \tensor \ff^+).(u \tensor v + \idem_1. u \tensor (\xi-1)\idem_1. v )
\\\nonumber
&= \xp. u \tensor v + \xp\idem_1. u \tensor (\xi-1)\idem_1. v
\\\nonumber
& \hspace{3em}+
(-1)^{|u|} \bar\z. u \tensor \xp. v
+ (-1)^{|u|} \bar\z \idem_1. u \tensor \xp(\xi-1)\idem_1. v ) \ ,
\end{align}
which can be checked to agree. The other cases are equally straightforward.
\end{proof}

Altogether, we have shown:

\begin{Prop}\label{prop:funSQ-mult}
	With the isomorphisms $\isoG_{U,V}$ as in~\eqref{eq:isoG}, the functor $\funSQ: \repS \to \repQ$ is multiplicative.
\end{Prop}

We conclude this section with a remark explaining how one may arrive at the expression \eqref{eq:isoG} for $\isoG_{U,V}$, and how it leads to the coproduct $\copS$ defined in the previous section.

\begin{Rem}\label{rem:origin-DeltaS}
Suppose we did not already know the coproduct $\copS$ and the isomorphisms $\isoG_{U,V}$. To find them, we start from the assumption that $\LL$ is group-like, $\copS(\LL)=\LL\tensor\LL$. Equivalently, we could assume that the tensor product on $\repS$ determined by $\copS$ is compatible with the $\oZ_2$-grading of $\repS$  in \eqref{eq:Z2-grading-repS}.

Given the definition of the functor $\funSQ$, we can now compare the $\K$-action on $\funSQ(U\svtensor V)$ and $\funSQ(U)\Qtensor\funSQ(V)$, as defined in \eqref{eq:funSQ}. This is done in the following table, where it is indicated whether $U,V$ are taken from $\repS_0$ or $\repS_1$, and where $u \in U$, $v \in V$:
\begin{equation*}
\raisebox{.7em}{
\begin{tabular}{ccccc}
 $U$ & $V$ & $\K.u$ & $\K.v$ & $\K.(u \otimes v)$
\\[.3em]
 $0$ & $0$ &
$\omega_U(u)$ & $\omega_V(v)$ & $\omega_{U \otimes V}(u \otimes v)$
\\[.3em]
 $0$ & $1$ &
$\omega_U(u)$ & $\rmi \,\omega_V(v)$ & $\rmi \,\omega_{U \otimes V}(u \otimes v)$
\\[.3em]
 $1$ & $0$ &
$\rmi \,\omega_U(u)$ & $\omega_V(v)$ & $\rmi \,\omega_{U \otimes V}(u \otimes v)$
\\[.3em]
 $1$ & $1$ &
$\rmi \,\omega_U(u)$ & $\rmi \,\omega_V(v)$ & $\omega_{U \otimes V}(u \otimes v)$
\end{tabular}}
\end{equation*}
Since we need $\isoG_{U,V}(\K.(u \otimes v)) = \K. \isoG_{U,V}(u \otimes v)$ (and since $\omega_{U \otimes V}(u\otimes v) = \omega_U(u) \otimes \omega_V(v)$ and $\K$ is group-like in $\Q$), we see that in all sectors but the last, we can simply take $\isoG_{U,V}(u \otimes v) = u \otimes v$. In the \textbf{11}-sector we need an extra minus sign, which suggests to take $\isoG_{U,V}$ to be the action of an odd element. One possibility (amongst many) is to choose $u \otimes v \mapsto u \otimes \xi.v$ with $\xi$ as in~\eqref{eq:isoG}.
This map is invertible since $\xi^2=\idem_1$ is the identity in $\Salg_1$.

Now that we have our ansatz for $\isoG_{U,V}$, we can ask for a coproduct on $\Salg$ such that $\isoG_{U,V}$ is a $\Q$-module intertwiner.
That is, we require commutativity of the diagram, for all $U,V\in\repS$,
\begin{equation}
\xymatrix@R=29pt@C=42pt{
&\Q\tensor \bigl(U\tensor_{\repS}V\bigr)\ar[d]_{\rho^{\Salg}}\ar[r]^{\id\tensor\isoG_{U,V}}
&\Q\tensor \bigl(U\Qtensor V\bigr)\ar[d]^{\rho^{\Q}}\\
&U\tensor_{\repS}V\ar[r]^{\isoG_{U,V}}&U\Qtensor V
}
\end{equation}
where $\rho^{\Salg}$ and $\rho^{\Q}$ are the action maps that use the coproducts $\copS$ and $\Delta$, respectively.
Note that the parity has to be taken into account while applying $\rho^{\Salg}$: $(a_1\tensor a_2).(u\tensor v) = (-1)^{|a_2||u|}a_1.u\tensor a_2.v$, for homogeneous $a_1, a_2\in\Salg$ and $u\in U$, $v\in V$,
	while no parity signs appear for $\rho^{\Q}$.
The diagram states that  for each $a \in \Q$, $u \in U$, $v \in V$ we require $\isoG_{U,V}(a.(u \otimes v)) = a.\isoG_{U,V}(u \otimes v)$. 
 For $a = \K^2$ (defining $\copS(\LL)$), this holds by construction. For $a = \ff^\pm$, one can choose $U=V=\Salg$, $u=v=\one$ and use $(\Gamma_{U,V})^2 = \id$ to get
\begin{equation}
	\copS(\xpm) = \Gamma_{\Salg,\Salg}\big(\Delta(\ff^\pm) \cdot \Gamma_{\Salg,\Salg}(\one \tensor \one)\big) \ ,
\end{equation}
which reduces to \eqref{copS}.
\end{Rem}

\section{An equivalence from $\catSF$ to $\repS$}\label{sec:SF-RepS}

\subsection{The functor $\funD$}\label{sec:funD}
We begin with defining an equivalence of $\oC$-linear abelian categories $\funD: \catSF \to \repS$.
	Recall the decomposition 
$\catSF=\catSF_0\oplus\catSF_1$ from~\eqref{eq:catSF-dec}
and  
$\repS = \repS_0 \oplus \repS_1$ from \eqref{eq:Z2-grading-repS}. 
The functor $\funD$  is defined differently in the two sectors. On $\catSF_0$ we take
\begin{equation}
	\funD_0: ~~\catSF_0 \to \repS_0 \ ,
 \quad \funD_0(U)=U \quad\text{where}\quad
\xp .u = a_1. u~,~~\xm. u = a_2. u \quad (u\in U) \ ,
\end{equation}
and $\funD_0(f)=f$ on morphisms.
Here we use the isomorphism $\xp\mapsto a_1$, $\xm \mapsto a_2$ of the Lie super algebras $\Salg_0$ and $\algGr$, see Section~\bref{sec:cat-C}.

Define the non-central idempotent
\begin{equation}\label{eq:Bel}
	\Bel := \xm \xp\idem_1 ~ \in  \Salg_1 \ .
\end{equation}
It generates a projective
$\Salg_1$-submodule $\B\subset\Salg_1$:
\begin{equation}\label{eq:B}
\B := \Salg\Bel = \Salg_1\Bel =
	\oC \,\xm \xp\idem_1 \, \oplus \, \oC\, \xp\idem_1 \ .
\end{equation}
We note that the image $\funSQ(\B)$ is the projective $\Q$-module  $\XX^{+}_{2}$ defined in Section~\bref{sec:repQ}.
The second component of $\funD$ is defined as
\begin{equation}\label{eq:fun-D1-def}
	\funD_1:~~ \catSF_1 \to \repS_1 \ , \quad \funD_1(U) = \B\svtensor U \ ,
\end{equation}
where the $\Salg_1$-action on $\funD_1(U)$ is the left $\Salg_1$-action on $\B$, and where $U\in\catSF_1=\svect$.
	On morphisms we set $\funD_1(f) = \id_\B \tensor f$.

\begin{Prop}\label{lem:D-equiv}
	The functor $\funD$ is an equivalence of $\oC$-linear categories.
\end{Prop}

\begin{proof}
We will construct a functor $\funE:\repS\to\catSF$ which is inverse to $\funD$ separately in the two sectors.

Given $V\in\repS_0$, the $\algGr$-module $\funE(V)$ has $V$ as the underlying super-vector space with the action of $a_1$ and $a_2$ given by the action of $\xp$ and $\xm$, respectively. On morphisms $f$ in $\repS_0$ we set $\funE(f) = f$. 

For a given object $V\in\repS_1$, we set $\funE(V)=\B^*\tensor_{\Salg_1} V$ and $\B^*$ is the right $\Salg_1$-module with action
$f(\cdot)a=f(a\cdot)$, for any $\oC$-linear map $f: \B\to\oC$ and $a\in\Salg_1$. On morphisms $f$ in $\repS_1$ we set $\funE(f) = \id_{\B^*}\tensor_{\Salg_1} f$.

To see that the functor $\B^*\tensor_{\Salg_1} (-)$ 
is invertible, we need one more observation:
the algebra $\Salg_1$ is isomorphic to $\B\tensor_{\svect}\B^*$ as a $(\Salg_1,\Salg_1)$-bimodule. The isomorphism can be described explicitly on a basis of $\Salg_1$ as
\begin{equation*}
 \idem_1 \mapsto \Bel\tensor \Bel^* + \Cel\tensor \Cel^*,\quad
 \xp\idem_1 \mapsto \Cel\tensor \Bel^*,\quad
 \xm\idem_1 \mapsto \Bel\tensor \Cel^*,\quad
 \xp\xm\idem_1 \mapsto \Cel\tensor \Cel^*,
\end{equation*}
where $\Bel =  \xm \xp\idem_1$ and $\Cel := \xp \Bel = \xp \idem_1$ are a basis of the left module $\B$,
while  $\Bel^*$ and $\Cel^*$ give the corresponding dual basis  in the right module $\B^*$.
We  have 
  $\om(\Bel^*)=\Bel^*$
 and $\Bel^*. \xp = 0$, $\Bel^*. \xm = \Cel^*$, etc.
It follows from this isomorphism that, conversely, 
$\B^*\tensor_{\Salg_1}\B\cong\oC^{1|0}$. 

It is now clear that the compositions $\funE\circ\funD$ and $\funD\circ\funE$ are naturally isomorphic
to the identity functors $\one_{\catSF}$ and $\one_{\repS}$, respectively.
\end{proof}

\subsection{$\funD$ as a multiplicative functor}

Below we will make an ansatz for isomorphisms
\begin{equation}\label{eq:isoD-source-target}
\isoD_{U,V}: \funD(U\Ctensor V) \to D(U)\Stensor D(V) \ .
\end{equation}
We do not know a good motivation for this ansatz. It was found by starting
from a  more general ansatz for $\isoD_{U,V}$ with free parameters
 and then using computer algebra
to fix a solution which
transports the braiding in $\catSF_0$ to the standard R-matrix for $\Q_0$ (see Section~\bref{sec:trans-braid} and Remark~\bref{rem:trans-braid-0}) and
makes the associator of $\catSF$ look as simple as possible after it is transported to $\repS$.

\begin{figure}[bt]
\begin{center}
\begin{tabular}{l@{\hspace{6em}}l}
\raisebox{4.5em}{\bf 00} \quad
\raisebox{-0.5\height}{\includegraphics[scale=0.3]{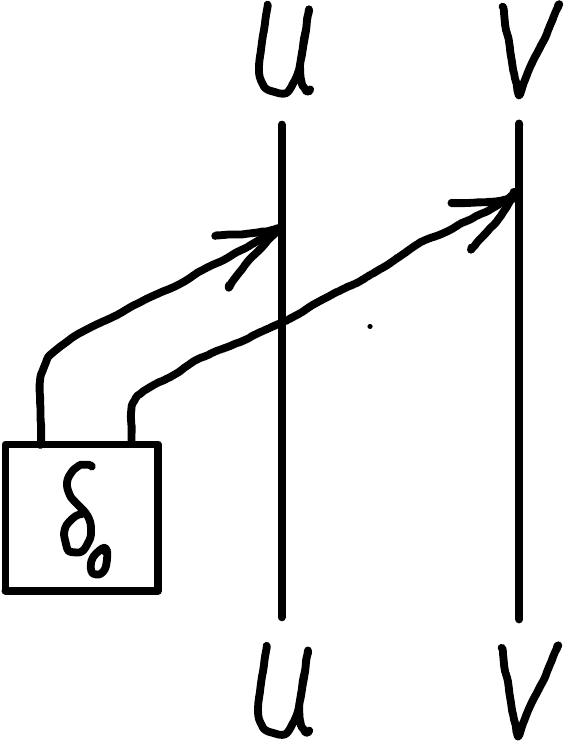}}
&
\raisebox{4.5em}{\bf 01} \quad
\raisebox{-0.5\height}{\includegraphics[scale=0.3]{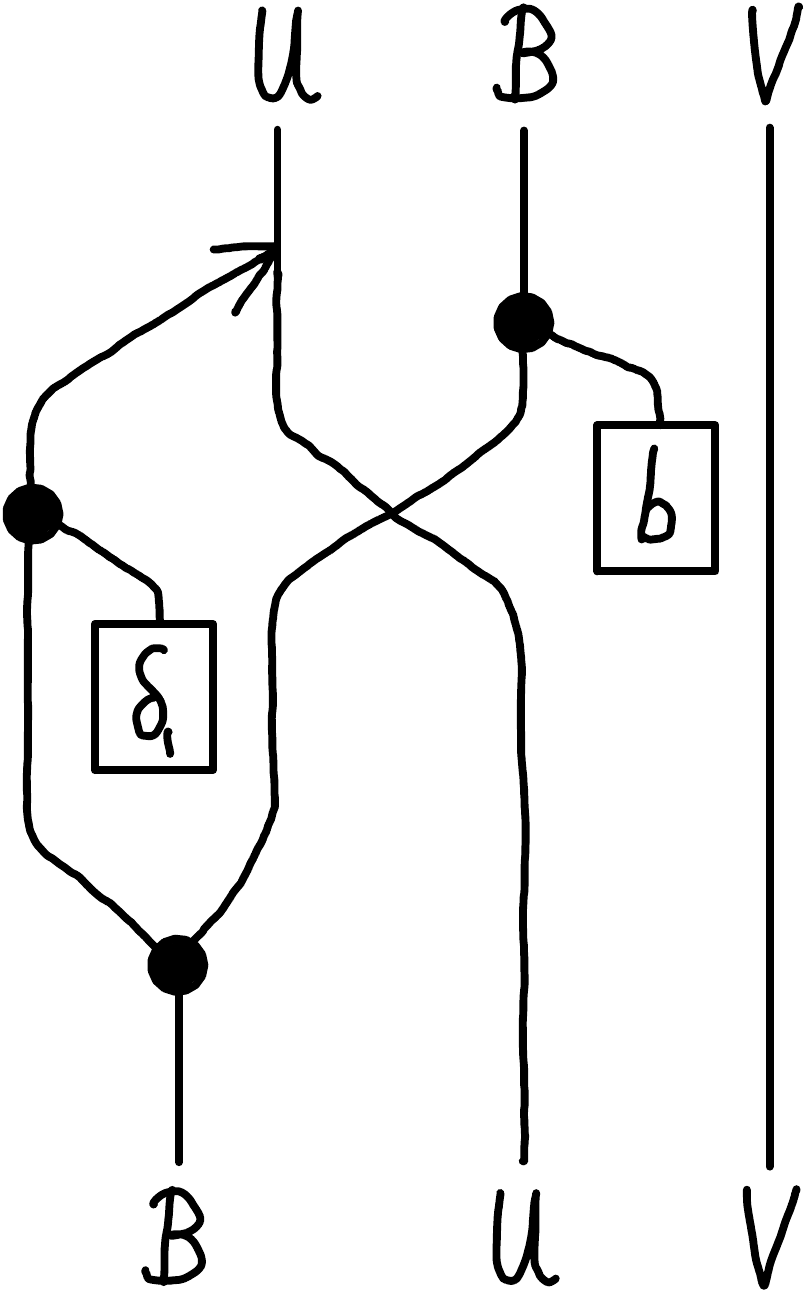}}
\\[6em]
\raisebox{5.5em}{\bf 10} \quad
\raisebox{-0.5\height}{\includegraphics[scale=0.3]{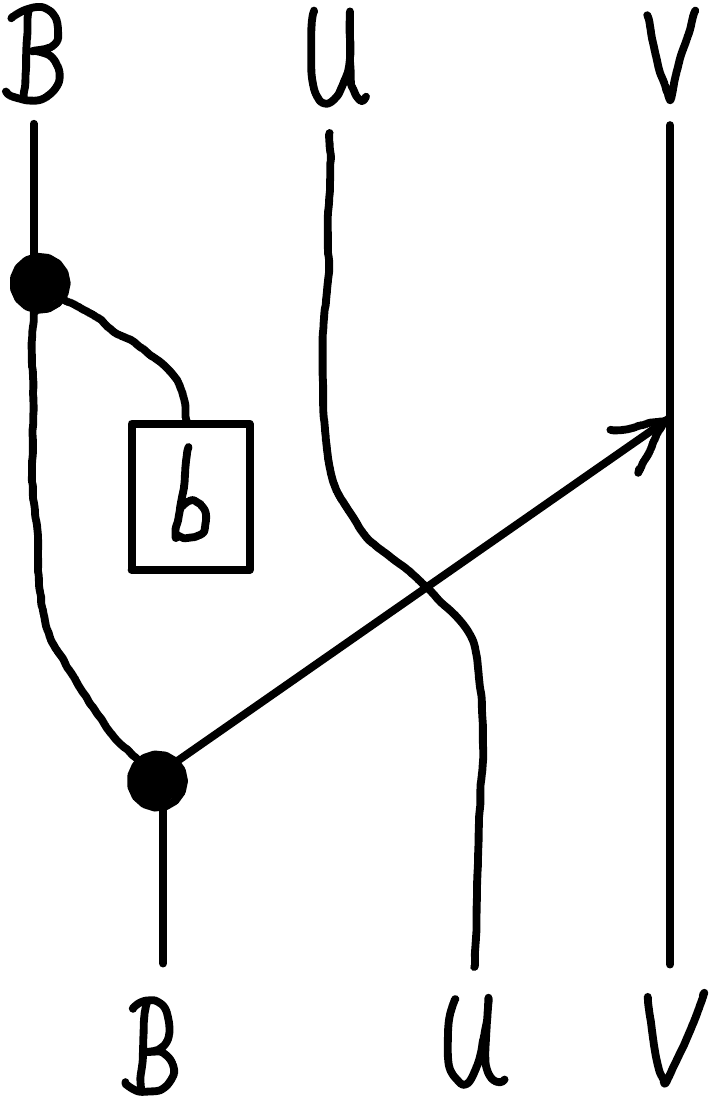}}
&
\raisebox{5.5em}{\bf 11} \quad
\raisebox{-0.5\height}{\includegraphics[scale=0.3]{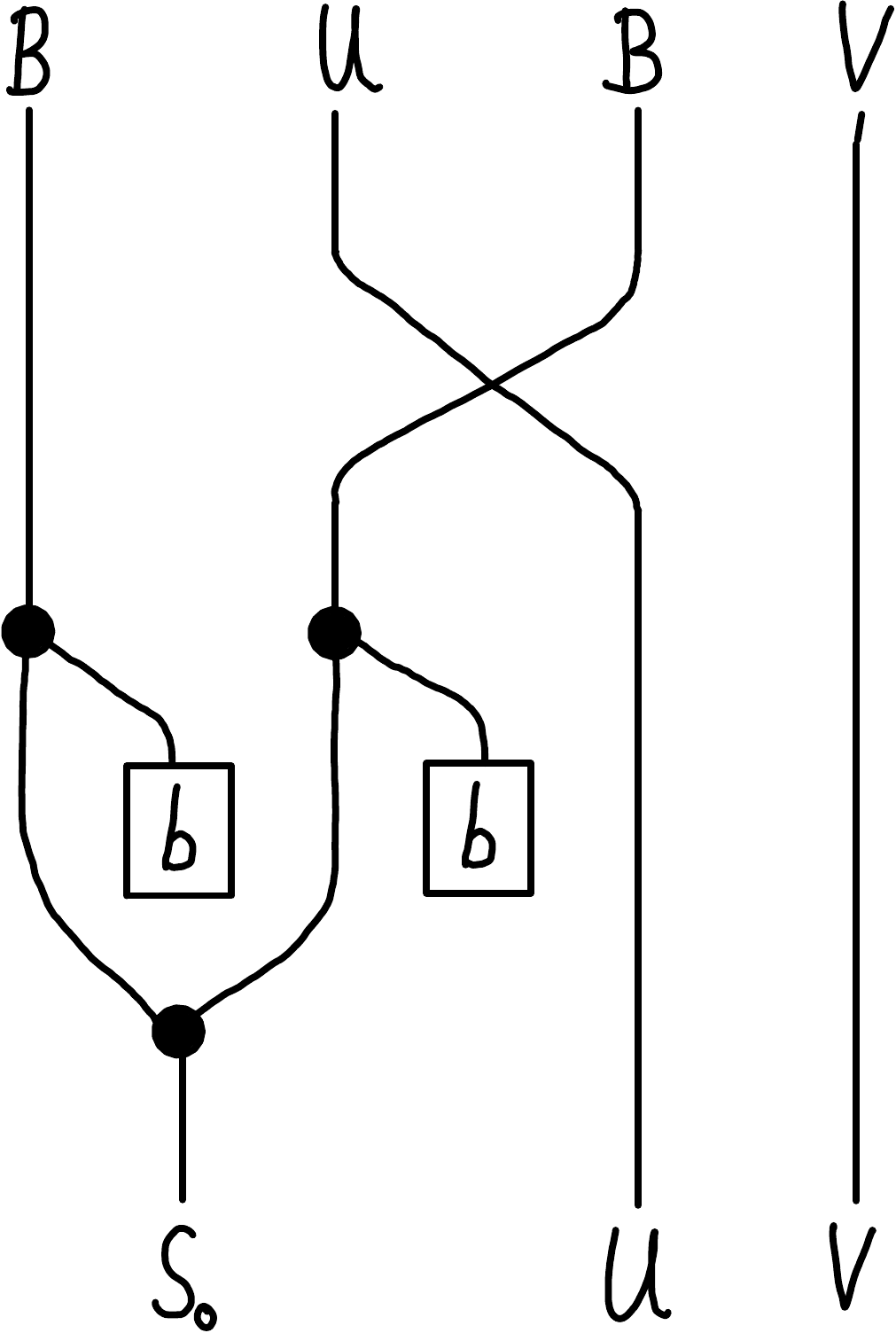}}
\end{tabular}
\end{center}
\caption{String diagram notation for $\isoD_{U,V}$ in the four sectors. Diagrams are read from bottom to top. Since in this paper string diagrams
are only used occasionally to make some formulas more readable,
we do not detail our conventions, but instead refer to \cite[Sect.\,2.1]{Davydov:2012xg}.}
\label{fig:Delta_UV}
\end{figure}

The resulting ansatz is defined sector by sector. Before we give it, we need a bit of notation. Firstly, the underlines in the source and target for $\isoD_{U,V}$ below indicate on which tensor factors $\Salg$ acts. For example, $\underline U \tensor \underline \B \tensor V$ means $\Salg$ acts (via the coproduct) on $U$ and $\B$, but not on $V$. Next, $\mu^\Salg$ stands for the multiplication $\Salg \tensor \Salg \to \Salg$, and for an $\Salg$-module $U$ we denote the action map by $\rho^U : \Salg \tensor U \to U$. We will also need the constants
\begin{equation}\label{del0}
	\del_{0}=\one\tensor\one + \xm\tensor\xp \in \Salg \tensor \Salg
\quad , \quad
\del_{1} = \one + \xp\xm \in \Salg \ .
\end{equation}
Note that they are multiplicatively invertible with inverses
\begin{equation}
	\del_{0}^{-1}=\one\tensor\one - \xm\tensor\xp \in \Salg \tensor \Salg
\quad , \quad
\del_{1}^{-1} = \one - \xp\xm(\idem_0+\tfrac12 \idem_1) \in \Salg \ .
\end{equation}
The right multiplication by $a \in \Salg_0$ will be denoted by
\begin{equation}
R_a : \Salg \to \Salg \quad , \quad R_a = \mu^\Salg \circ (\id_\Salg \tensor a) \ .
\end{equation}
Since $a$ is parity-even, this is indeed a morphism in $\svect$, and, since the multiplication is on the right, it is also a morphism in $\repS$. Note that
the image of $R_\Bel$, with $\Bel$ as in \eqref{eq:Bel}, is precisely $\B$, and we will use $R_\Bel$ to project from $\Salg$ to $\B$ in a way compatible with the $\Salg$-action.

In Figure \ref{fig:Delta_UV} we express the formulas for $\isoD_{U,V}$ given below in terms of string diagrams.
\begin{itemize}

\item \textbf{00} sector: $\isoD_{U,V}: \underline U\tensor \underline V\to \underline U \tensor \underline V$ is given by
\begin{equation}\label{eq:isoD00}
\isoD_{U,V} = \big(\rho^U \tensor \rho^V\big) \circ \big(\id_{\Salg} \tensor \sflip_{\Salg,U} \tensor \id_V\big) \circ \big(\delta_0 \tensor \id_U \tensor \id_V\big) \ ,
\end{equation}
where $\sflip$ is defined in \eqref{sflip}.
\item \textbf{01} sector: $\isoD_{U,V}: \underline\B\tensor U\tensor V\to \underline U \tensor \underline\B\tensor V$ is given by
\begin{equation}\label{eq:isoD01}
\isoD_{U,V} =
\big(\rho^U \tensor R_\Bel \tensor \id_V \big) \circ
 \big(R_{\delta_1} \tensor \sflip_{\Salg,U} \tensor \id_V\big)
\circ \big( \copS \tensor \id_U \tensor \id_V \big) \ .
\end{equation}
Since $\B \subset \Salg$, it makes sense to apply $\copS$ to elements in $\B$. The right multiplication with $\Bel$ projects from $\Salg$ to $\B$, so that the target is indeed $U \tensor \B \tensor V$.

\item \textbf{10} sector: $\isoD_{U,V}: \underline\B\tensor U\tensor V\to \underline\B \tensor U\tensor \underline V$  is given by
\begin{equation}\label{eq:isoD10}
\isoD_{U,V} =
\big(R_\Bel \tensor \id_U \tensor \rho^V \big) \circ
 \big(\id_\Salg \tensor \sflip_{\Salg,U} \tensor \id_V\big)
\circ \big( \copS \tensor \id_U \tensor \id_V \big) \ .
\end{equation}

\item \textbf{11} sector: $\isoD_{U,V}: \underline{\Salg_0}\tensor U\tensor V\to \underline\B\tensor U \tensor \underline\B\tensor V$ is given by
\begin{equation}\label{eq:isoD11}
\isoD_{U,V} =
 \big(\id_\B \tensor \sflip_{\B,U} \tensor \id_V\big)
\circ \big( [(R_\Bel \tensor R_\Bel) \circ \copS] \tensor \id_U \tensor \id_V \big) \ .
\end{equation}
Here, in the source $\Salg$-module we have identified $\Salg_0$ and $\algGr$.
\end{itemize}

\begin{Lemma}
The linear maps $\isoD_{U,V}$ are intertwiners of $\Salg$-modules.
\end{Lemma}

\begin{proof}
In sector {\bf 00}, we have to verify that for all $a \in \Salg$ we have $(\idem_0 \tensor \idem_0) \cdot \copS(a) \cdot  \delta_0 = (\idem_0 \tensor \idem_0)\cdot \delta_0 \cdot \copS(a)$. Here the idempotents $\idem_0$ appear because $\Salg \tensor \Salg$ acts on two representations from $\repS_0$ (without the factor $\idem_0 \tensor \idem_0$ the identity is false).
For $a=\LL$, the above identity is clear, and for $a=\xpm$, it follows from  a short calculation.

In the other three sectors, the intertwiner property is even more direct: its follows since $\copS$ is an algebra map, and since the right-multiplications $R_{\delta_1}$ and $R_\Bel$ are left-module intertwiners. (For those versed in string diagrams, the intertwining property will be obvious from Figure~\ref{fig:Delta_UV}.)
\end{proof}

\begin{Lemma}
The $\isoD_{U,V}$ are isomorphisms.
\end{Lemma}

\begin{proof}
We proceed sector by sector.
Since $\delta_0$ has a multiplicative inverse, it is clear that $\isoD_{U,V}$ is invertible in sector {\bf 00}. In sectors {\bf 01} and {\bf 10} the inverse maps are given by, in string diagram notation:
\begin{equation*}
\raisebox{4.5em}{\bf 01} \quad
\raisebox{-0.5\height}{\includegraphics[scale=0.3]{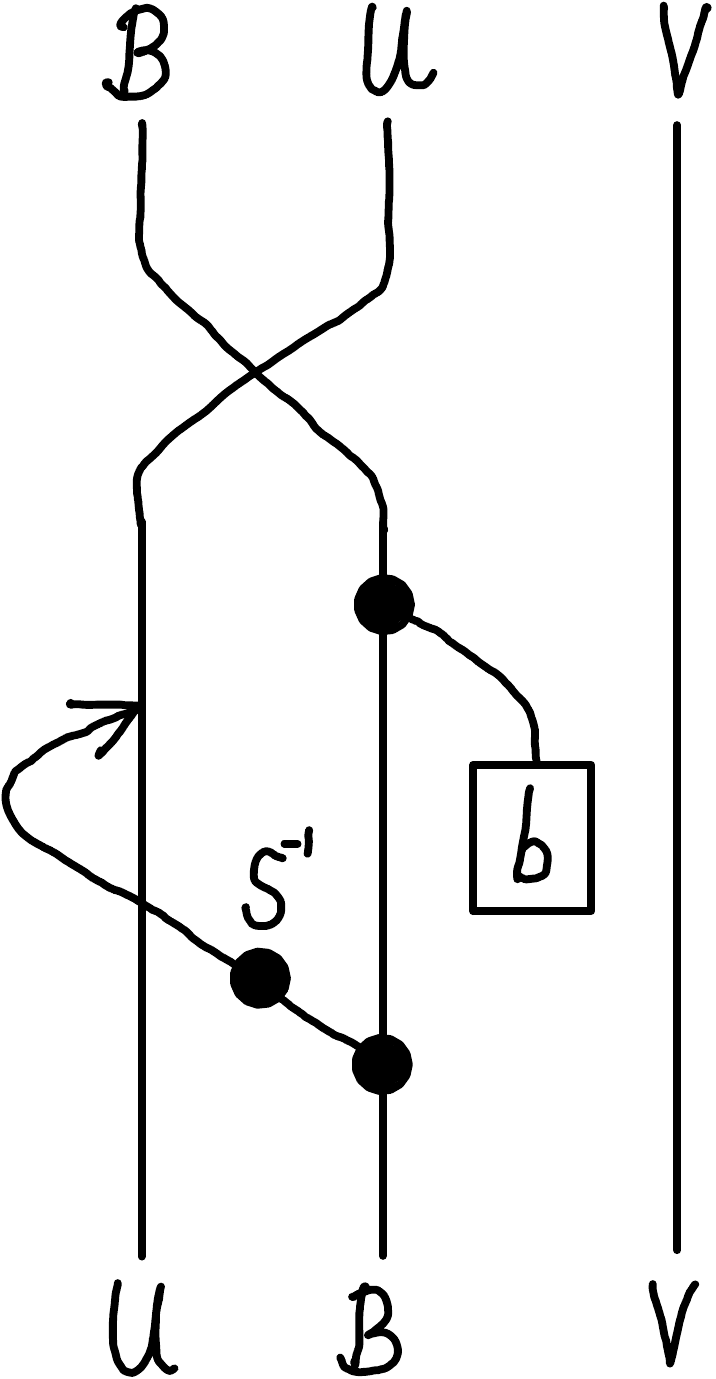}}
\qquad \qquad
\raisebox{4.5em}{\bf 10} \qquad
\raisebox{-0.5\height}{\includegraphics[scale=0.3]{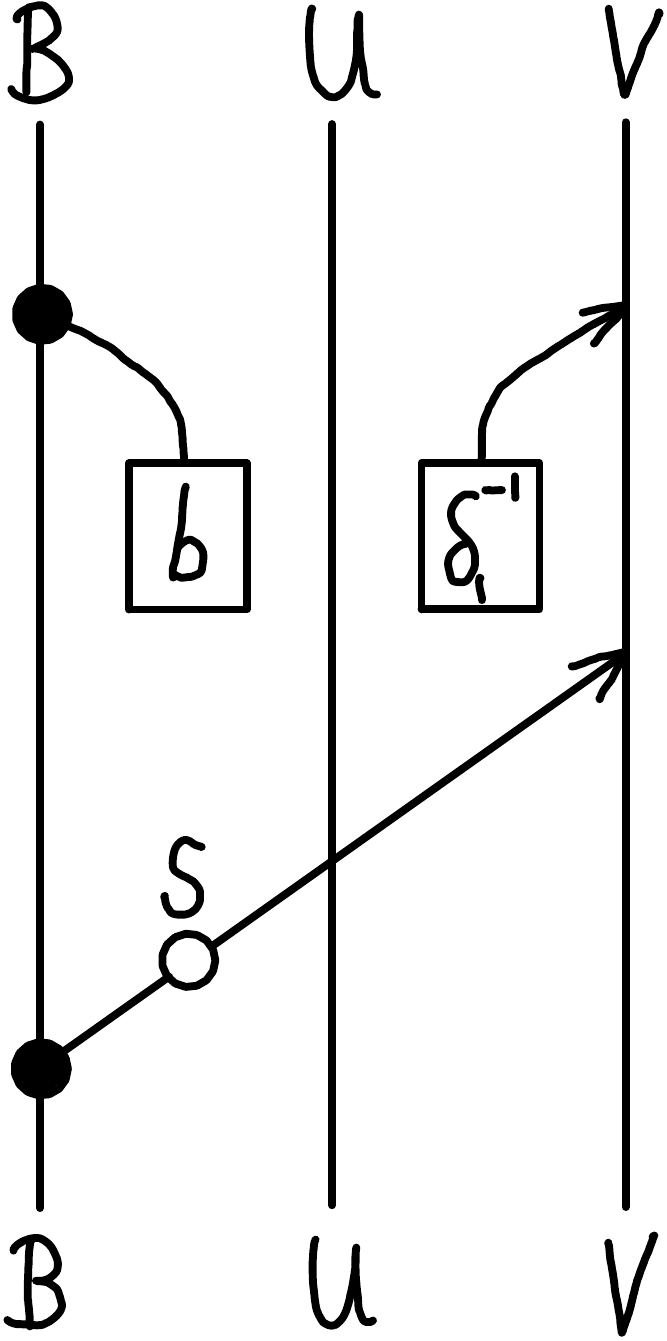}}
\end{equation*}
For example, in sector {\bf 10}, compute $\isoD^{-1}_{UV} \circ \isoD_{UV}$ by first moving the $\Bel$ in $\isoD_{UV}$ past the coproduct in $\isoD^{-1}_{UV}$. Then use the identity 
$(\Bel \otimes \one) \cdot ( (\id \otimes S) \circ \isoD(\Bel)) \cdot (\Bel \otimes \one) = \Bel \otimes (\delta_1 \idem_0)$
and $\mu^\Salg \circ (S \otimes \id) \circ \copS = \one \circ \eps$ (see \cite[Lem.\,2.3]{Davydov:2012xg}). 
In this calculation, one needs coassociativity of $\copS$, which holds here as we are using it only in the sector $S_1 \to S_1 \otimes S_0 \otimes S_0$, where $\Sas$ is the identity.
Since the underlying vector spaces are finite-dimensional, it is enough to check $\isoD^{-1}_{UV} \circ \isoD_{UV} = \id$. The computation in sector ${\bf 01}$ is similar. 

In sector {\bf 11}, we establish that the map $\isoD_{U,V}$  is surjective by a direct computation.
As in the proof of Proposition \bref{lem:D-equiv}, we first fix a basis in $\B$ as  $\B=\oC \Bel\oplus\oC\Cel$ and $\Cel=\xp.\Bel$. Next, note the equation 
\begin{equation}
\copS(\xpm) \cdot \idem_1\tensor\idem_1 = (\xpm\tensor\one - \rmi \one\tensor\xmp) \cdot \idem_1\tensor\idem_1 \ ,
\end{equation}
which allows one  to compute  
 the image $\copS(\Salg_0) \cdot \Bel\tensor\Bel$ easily:
\begin{align*}
\copS(\xp\idem_0) \cdot \Bel\tensor\Bel &= \Cel\tensor\Bel,\quad
&\copS(\idem_0) \cdot \Bel\tensor\Bel &= \Bel\tensor\Bel,\\
\copS(\xm\idem_0) \cdot \Bel\tensor\Bel &=-\rmi  \Bel\tensor\Cel, \quad
&\copS(\xm\xp\idem_0) \cdot \Bel\tensor\Bel &= \Bel\tensor\Bel + \rmi\Cel\tensor\Cel.
\end{align*}
This image is the  super vector space $\B\tensor_{\svect}\B$. Since the dimension of  $\Salg_0$ and $\B\tensor_{\svect}\B$ is four, surjectivity of $\isoD_{U,V}$ implies that it is an isomorphism. 

We thus have that $\isoD_{U,V}$ is an isomorphism in all four sectors.
\end{proof}

That $\isoD_{U,V}$ is natural in $U$ and $V$ is immediate. Thus, together with the two previous lemmas we see that \eqref{eq:isoD-source-target} is a natural family of isomorphisms, thus proving the following proposition.

\begin{Prop}
	With the isomorphisms $\isoD_{U,V}$ as in~\eqref{eq:isoD00}--\eqref{eq:isoD11}, the functor $\funD: \catSF \to \repS$ is multiplicative.
\end{Prop}

\section{Transport of associator, braiding and ribbon twist from $\catSF$ to $\repQ$}\label{sec:proof}

\newcommand{\tensC}{\tensor_{\cat}}
\newcommand{\assC}{\alpha^{\cat}}
\newcommand{\tensD}{\tensor_{\catD}}
\newcommand{\assD}{\alpha^{\catD}}

\subsection{Transporting associators along multiplicative functors}\label{sec:trans-assoc}

We now describe how to transport associators along multiplicative equivalences. Let $\cat$ be a monoidal category with a tensor-product functor $\tensC$ and an associator $\assC$, and $\catD$ a category  with a tensor-product functor $\tensD$ which is not equipped with an associator. Consider a multiplicative equivalence $\fun : \cat \to \catD$ with family of isomorphism $\{\Theta_{U,V}: \fun(U\tensC V)~\to~\fun(U)~\tensD \fun(V)\}$. We now seek natural isomorphisms $\alpha_{X,Y,Z}^\catD$, for $X,Y,Z \in \catD$, such that for all $U,V,W \in \cat$, the diagram
\begin{equation}\label{eq:transport-assoc-diag}
\xymatrix@R=32pt@C=72pt{
\fun \bigl(U\tensC (V\tensC W)\bigr)\ar[d]_{\Theta_{U,V\tensC W}}\ar[r]^{\fun(\assC_{U,V,W})}
&\fun \bigl((U\tensC V)\tensC W\bigr)\ar[d]^{\Theta_{U\tensC V,W}}\\
\fun (U)\tensD \fun(V\tensC W)\ar[d]^{\id\tensor\Theta_{V,W}}
&\fun (U\tensC V)\tensD \fun(W)\ar[d]^{\Theta_{U,V}\tensor\id}\\
\fun (U)\tensD\bigl(\fun(V)\tensD\fun(W)\bigr)\ar[r]^{\assD_{\fun(U),\fun(V),\fun(W)}}
&\bigl(\fun (U)\tensD \fun(V)\bigr)\tensD\fun(W)
}
\end{equation}
commutes. Such an $\alpha^\catD$ exists (use the functor inverse to $\fun$ to define it), is unique (since $\alpha^\catD$ is natural and $\fun$ is essentially surjective) and automatically satisfies the pentagon condition (as $\alpha^\cat$ does and the $\Theta_{U,V}$ are isomorphisms). The unit isomorphisms can be transported in the same way, turning $\catD$ into a monoidal category. By construction, for this monoidal structure on $\catD$, the multiplicative equivalence $\fun$ becomes monoidal. 

\medskip

Starting from the monoidal category $\catSF$, we use the above diagram to transport the monoidal structure from $\catSF$ to $\repS$ via $\funD$, and then further to $\repQ$ via $\funSQ$.

\subsection{Transporting the associator from $\catSF$ to $\repS$}\label{rem:Lambda-from-SF}
The coassociator $\Sas$ for $\Salg$ given in Section \bref{sec:S} was in fact computed via the method described above.
That is, the associator $\assocS$ in $\repS$ was computed from the multiplicative functor $\funD$ using 
the family $\isoD_{U,V}$ defined in~\eqref{eq:isoD00}-\eqref{eq:isoD11} and the associator $\alpha^{\catSF}_{U,V,W}$ of $\catSF$ as given in Section~\bref{sec:assoc-SF}.
	It turns out\footnote{In the category of vector spaces this would be automatic, but in super-vector spaces, the associator could involve in addition the parity involution $\omega$.} that the associator $\alpha^{\catSF}_{U,V,W}$ can be defined via the action of a (necessarily even)  element $\Sas \in \Salg \tensor \Salg \tensor \Salg$,
\begin{equation}\label{eq:assocS}
 \Sas. (u\tensor v\tensor w) = \assocS(u\tensor v\tensor w)  \ .
\end{equation}
To determine $\Sas$ (or to verify the expression in \eqref{eq:Sas}, depending on the point of view), we evaluate the diagram \eqref{eq:transport-assoc-diag} in each of the eight sectors in turn.
That is, we write
\begin{equation}
\Sas = \sum_{a,b,c \in \{0,1\}} \tilde\Sas^{abc}
\qquad \text{where} \quad \tilde\Sas^{abc} \in \Salg_a \tensor \Salg_b \tensor \Salg_c 
\end{equation}
and solve the condition \eqref{eq:transport-assoc-diag} for each $\tilde\Sas^{abc}$ separately.
The relation to the $\Sas^{abc}$ entering the expression for $\Sas$ in \eqref{eq:Sas} is $\tilde\Sas^{abc} = \Sas^{abc} \cdot \idem_a \tensor \idem_b \tensor \idem_c$.

For example, in sector {\bf 000}, diagram \eqref{eq:transport-assoc-diag} reads
\begin{equation}\label{eq:Lam-diag-000}
\raisebox{-0.5\height}{\includegraphics[scale=0.3]{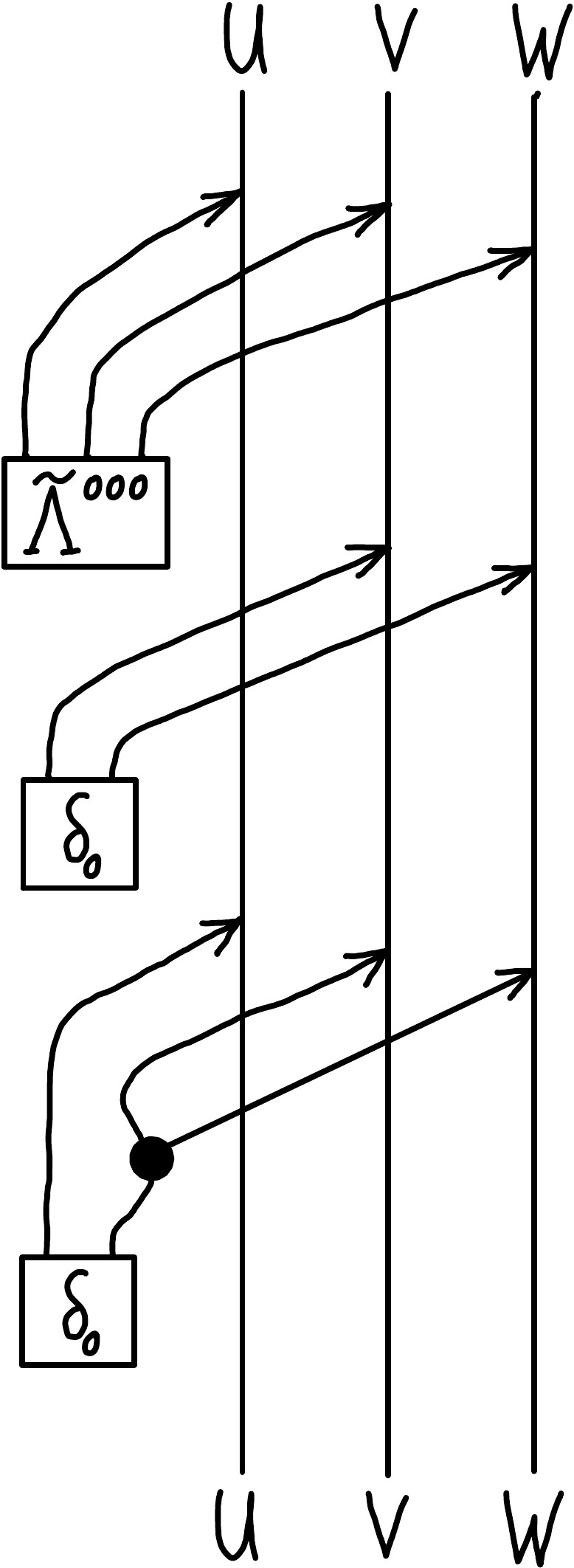}}
\qquad=\qquad
\raisebox{-0.5\height}{\includegraphics[scale=0.3]{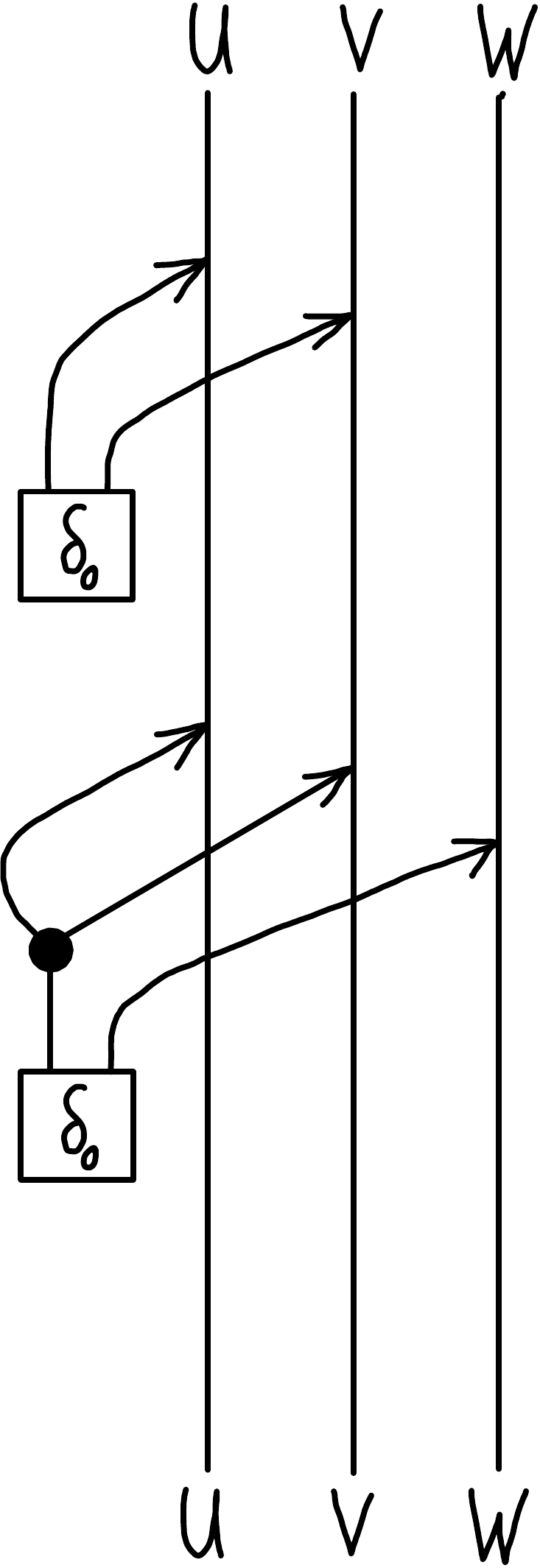}}
\end{equation}
Choosing $U=V=W=\Salg_0$ (recall that we identify $\Salg_0$ and $\algGr$) and evaluating on 
	$\idem_0 \tensor \idem_0 \tensor \idem_0$, 
we see that \eqref{eq:Lam-diag-000} implies
\begin{equation}\label{eq:Lam-diag-000r}
	\tilde\Sas^{000} \cdot (\one \tensor \delta_0) \cdot (\id \tensor \copS)(\delta_0)
	= 
	(\idem_0 \tensor \idem_0 \tensor \idem_0) \cdot
	(\delta_0 \tensor \one) \cdot (\copS \tensor \id)(\delta_0) \ .
\end{equation}
On the other hand, if \eqref{eq:Lam-diag-000r} holds, so does \eqref{eq:Lam-diag-000} for all $U,V,W \in \catSF_0=\repS_0$ by associativity of the $\Salg$-action.
Finally, as $\delta_0$ is multiplicatively invertible, the solution $\tilde\Sas^{000} \in \Salg_0 \tensor \Salg_0 \tensor \Salg_0$ to \eqref{eq:Lam-diag-000r} is unique and given by $\idem_0 \tensor \idem_0 \tensor \idem_0$, as already stated in \eqref{eq:Sas}.
The reasoning is the same in all the sectors as
defined by the projectors $\idem_{i}\tensor\idem_{j}\tensor\idem_{k}$, 
and we just state the condition analogous to \eqref{eq:Lam-diag-000r} 
in each sector in Table \ref{tab:Lam-eqn-sectors}.
On the one hand, these conditions determine $\Sas$ uniquely (by invertibility of $\isoD_{U,V}$), and on the other hand, they guarantee commutativity of the diagram \eqref{eq:transport-assoc-diag} for all $U,V,W$.

\begin{table}[bt]
\begin{align*}
&{\bf001}~:~&&  \tilde\Sas^{001} \cdot \bigl( (\id\tensor\copS)\circ(R_{\delta_1}\tensor R_{\Bel})\circ\copS(v)\bigr)\cdot(\one\tensor\delta_1\tensor\Bel)\\
&&&= (\idem_0\tensor\idem_0\tensor\one)\cdot(\delta_0\tensor\one)\cdot\bigl((\copS\tensor
\id)
\circ(R_{\delta_1}\tensor
	R_{\Bel}
)\circ\copS(v)\bigr)\\[.6em]
&{\bf010}~:~&&  \tilde\Sas^{010} \cdot \bigl( (\id\tensor\copS)\circ(R_{\delta_1}\tensor R_{\Bel})\circ\copS(v)\bigr)\cdot(\one\tensor\Bel\tensor\one)\\
&&&= 
 \bigl((\copS\tensor
	\id
 )\circ(R_{\Bel}\tensor
 	\id
 )\circ\copS(v)\bigr)\cdot(\delta_1\tensor\Bel\tensor\one)\cdot\gamma^{(13)}\\[.6em]
&{\bf100}~:~&&  \tilde\Sas^{100}\cdot (\one\tensor\delta_0)\cdot\bigl( (\id\tensor\copS)\circ(R_{\Bel}\tensor
	\id
)\circ\copS(v)\bigr)\\
&&&=\bigl((\copS\tensor
	\id
)\circ(R_{\Bel}\tensor
	\id
) \circ\copS(v)\bigr)\cdot(\Bel\tensor\one\tensor\one)\\[.6em]
&{\bf110}~:~&&  \tilde\Sas^{110} \cdot \bigl( (\id\tensor\copS)\circ(R_{\Bel}\tensor R_{\Bel})\circ\copS(h)\bigr)\cdot (\one\tensor\Bel\tensor\one)  \\
&&&= \bigl\{(\copS\tensor\id)\bigl(\delta_0\cdot\copS(h)\bigr)\bigr\}\cdot 
(\Bel\tensor\Bel\tensor\one)\\[.6em]
&{\bf101}~:~&&  \tilde\Sas^{101} \cdot \bigl\{ (\id\tensor\copS)\bigl( \copS(h)\cdot \Bel\tensor\Bel\bigr)\bigr\}\cdot (\one\tensor\delta_1\tensor \Bel)\\
&&&= (R_{\Bel}\tensor\mu^{\Salg}\tensor
	\id
)\circ(\copS\tensor
   \sflip_{\B,\Salg_0}
)\circ\bigl(\bigl\{ (R_{\Bel}\tensor R_{\Bel})\circ\copS\circ\mu^{\Salg}\bigr\}\tensor
	\id
\bigr) (h\tensor\gamma)\\[.6em]
&{\bf011}~:~&&  \tilde\Sas^{011} \cdot \bigl( (\id\tensor\copS)(\delta_0)\bigr)\cdot\bigl(\one\tensor(\copS(h)\cdot\Bel\tensor\Bel)\bigr)\\
&&&= \bigl\{\bigl(S^{-1}\circ\mu^{\Salg}\circ(\id\tensor S)\bigr)\tensor\id\tensor \id\bigr\}\\
&&&\qquad\qquad\circ \bigl\{\id\tensor \bigl((R_{\delta_1}\tensor R_{\Bel})\circ\copS\bigr)\tensor
	\id
\bigr\}
 \circ \bigl\{\id\tensor \bigl((R_{\Bel}\tensor R_{\Bel})\circ \copS\bigr)\bigr\}\circ \copS(h)\\[.6em]
&{\bf111}: \quad&& \tilde\Sas^{111} \cdot \bigl\{\bigl(\id\tensor(\copS\circ\mu^\Salg)\bigr)\bigl(\copS(v)\tensor h\bigr)\bigr\} \cdot \Bel\tensor\Bel\tensor\Bel\\
&&&=\bigl\{\bigl((R_{\Bel}\tensor R_{\Bel})\circ\copS\circ\mu^\Salg\bigr)\tensor
	\id
\bigr\}
\circ(
	\id
\tensor\sflip_{\B,\Salg_0})\circ\bigl\{\bigl((R_{\delta_1}\tensor R_{\Bel})\circ\copS(v)\bigr)\tensor\phi(h)\bigr\}
\end{align*}
\caption{Conditions determining 
$\tilde\Sas^{abc}\in\Salg_{a}\tensor\Salg_b\tensor\Salg_c$
	via an equation in $\Salg_{a}\tensor\Salg_b\tensor\Salg_c$. 
The condition for sector {\bf 000} is stated in \eqref{eq:Lam-diag-000r}. The above conditions have to hold for all $h\in\Salg_0$ and $v\in\B$.
In sectors {\bf 010} and {\bf 101}, $\gamma$ is given by 
$\gamma := \exp{C} = (\one\tensor\one + \xm\tensor \xp - \xp\tensor \xm  - \xp \xm\tensor \xp \xm)\cdot\idem_0\tensor\idem_0$ as in Section~\bref{sec:assoc-SF}; the notation $\gamma^{(13)}$ means that it acts on the first and third tensor factors. The map $\phi$ in sector {\bf 111} is given in \eqref{eq:SF-phi-def}.}\label{tab:Lam-eqn-sectors}
\end{table}

Since $\assocS$ satisfies the pentagon by construction, $\Sas$ satisfies the $3$-cocycle condition \eqref{eq:3-cocycle}. (We nonetheless checked this independently by computer algebra.)
	$\Sas$ is in addition counital, $(\id \tensor \eps \tensor \id)(\Lambda) = \one \tensor \one$.
We thus arrive at the following proposition.

\begin{Prop}\label{prop:catSF-repS-monoid}
The natural isomorphisms $\assocS$ from~\eqref{eq:assocS} with $\Sas$ as in \eqref{eq:Sas} define an associator on $\catD$. With respect to this associator, the equivalence $\funD: \catSF \to \repS$ with multiplicative structure $\Delta_{U,V}$ defined in~\eqref{eq:isoD00}-\eqref{eq:isoD11} is monoidal.
\end{Prop}

\subsection{Transporting the associator from $\repS$ to $\repQ$}
	We repeat the procedure in Section~\bref{sec:trans-assoc} and transport the associator $\assocS$ to an associator $\assocQ$ in $\repQ$ using 
	the  multiplicative equivalence $\funSQ:\repS\to\repQ$. 
Since $\repQ$ consists of (finite-dimensional) $\Q$-modules in vector spaces, the associator on $\repQ$ necessarily takes the form
\begin{equation}\label{eq:assoc-RepQ-via-Phi}
 \assocQ_{U,V,W}(u\tensor v\tensor w)
 = 
 \as. (u\tensor v\tensor w)  \ ,
\end{equation}
where $u,v,w$ are elements of the three $\Q$-modules $U,V,W$ and $\as \in \Q \tensor \Q \tensor \Q$. To compute $\Phi$, we can choose $U=V=W=\Q$ and evaluate on the element $\one \tensor\one\tensor\one$. 
	
In terms of the diagram \eqref{eq:transport-assoc-diag}, this means the following. Recall from the proof of Proposition~\bref{prop:G-equiv} the functor $\fun$ inverse to $\funSQ$. Let us abbreviate $\hat\Q := \fun(\Q)$. The $\Salg$-module $\hat\Q$ has parity involution given by \eqref{Z2-repQ}, $\LL$ acts by $\K^2$ and $\xpm$ act by $\ff^\pm$, see \eqref{eq:funSQ}. Diagram \eqref{eq:transport-assoc-diag} reads
\begin{equation}\label{eq:Phi-from-Lambda}
\big( (\isoG_{\hat\Q,\hat\Q} \tensor \id) \circ \isoG_{\hat\Q \tensor\hat\Q,\hat\Q}\big)\big(\Sas \,\hat. \,(\one \tensor \one \tensor \one)\big)
=
\as \cdot \big[ (\id\tensor\isoG_{\hat\Q,\hat\Q}) \circ \isoG_{\hat\Q,\hat\Q \tensor\hat\Q}(\one \tensor \one \tensor \one) \big] \ ,
\end{equation}
where $\isoG_{U,V}$ is the multiplicative structure from \eqref{eq:isoG} and the notation 
`\,$\hat.$\,' will be explained momentarily.

There are two slightly subtle points in evaluating \eqref{eq:Phi-from-Lambda}. Firstly, $\Sas$ acts on $\hat\Q \otimes \hat\Q \otimes \hat\Q$ via the symmetric braiding in $\svect$, i.e.\ with parity signs. We have written 
`\,$\hat.$\,'
instead of 
`\,.\,' in  \eqref{eq:Phi-from-Lambda} to stress this point. 
Secondly, $\one \in \hat\Q$ 
is not of definite parity (in particular, it is not parity-even), and, since $\hat\Q$ is different from $\Salg$, we cannot simplify $\Sas \,\hat. \,(\one \tensor \one \tensor \one)$ to $\Sas$ as might be suggested by the notation.
There is one other place where one has to be careful with parity signs, and this is the action of $\isoG_{\hat\Q,\hat\Q \tensor\hat\Q}$, which is, for $a,b,c \in \hat\Q$,
\begin{equation}
\isoG_{\hat\Q,\hat\Q \tensor\hat\Q}(a \tensor b \tensor c) = 
a \tensor b \tensor c + 
\idem_1. a \tensor \big[ \copS((\xi-1)\idem_1) \,\hat.\, (b \otimes c) \big] \ .
\end{equation}
Taking all this into account, the unique solution to \eqref{eq:Phi-from-Lambda} can be obtained with the help of computer algebra to be
\begin{multline} \label{eq:Phi-via-f+-}
\as = \idem_0\tensor\idem_0\tensor\idem_0 + \idem_0\tensor\idem_0\tensor\idem_1 + \idem_1\tensor\idem_0\tensor\idem_0
+\idem_0\tensor\idem_1\tensor\idem_1
+ \idem_1\tensor\idem_1\tensor\idem_0 \\
+ \as^{010}\idem_0\tensor\idem_1\tensor\idem_0
+ \as^{101}\idem_1\tensor\idem_0\tensor\idem_1
+ \as^{111}\idem_1\tensor\idem_1\tensor\idem_1\ ,
\end{multline}
where its non-trivial components are factorised as
\begin{align*}
\as^{010} =~& \Bigl(\one\tensor\one\tensor\one
    + (1 + \rmi) \ff^+\K\tensor\K\tensor\ff^-\Bigr)\Bigl(\one\tensor\one\tensor\one + (1 - \rmi) \ff^-\K\tensor\K\tensor\ff^+\Bigr)\ ,\\
\as^{101} =~&\Bigl( \one\tensor\one\tensor\one
+ (1+\rmi) \one\tensor\ff^+\K\tensor\ff^-
+ (1 - \rmi) \ff^-\K\tensor\ff^+\tensor\one
\Bigr)\\
~&\times\Bigl( \one\tensor\one\tensor\one
 + (1 + \rmi)\ff^+\K\tensor\ff^-\tensor\one
 + (1 - \rmi)\one\tensor\ff^-\K\tensor\ff^+ \Bigr)\one\tensor\K\tensor\one\ ,\\
\as^{111} =~&  \ffrac{\beta^2}{\rmi}
\Bigl( \one\tensor\one\tensor\one
+(\rmi-1)\bigl(\one\tensor\ff^+\K\tensor\ff^-
 +  \ff^+\K\tensor\K\tensor\ff^-
 - \ff^+\K\tensor\ff^-\tensor\one
  + \one\tensor\ff^-\ff^+\tensor\one\bigr)
\Bigr)\\
~&\times\Bigl( \one\tensor\one\tensor\one
 - (\rmi-1)\bigl( \one\tensor\ff^-\K\tensor\ff^+
 + \ff^-\K\tensor\K\tensor\ff^+
   -  \ff^-\K\tensor\ff^+\tensor\one
   - \one\tensor\ff^-\ff^+\tensor\one\bigr)
  \Bigr)\\
  ~&\times\Bigl( \one\tensor\one\tensor\one
   - 2 \one\tensor\ff^-\ff^+\tensor\one
   \Bigr)
  \one\tensor\K\tensor\one \ .
 \end{align*}
Above, we have expressed $\Phi$ in terms of the generators $\ff^\pm$ and in a factorised form, but one can check that it is equal to \eqref{eq:as-intro}.

By construction, $\Phi$ satisfies the compatibility condition
\begin{equation}
(\Delta\tensor\id)(\Delta(x)) = \as \cdot  \bigl( (\id\tensor\Delta)(\Delta(x))\bigr) \cdot \as^{-1} ~~,\quad x\in\Q \ ,
\end{equation}
with the coproduct of $\Q$, as well as the cocycle condition \eqref{eq:3-cocycle}.
Since the equivalences we use preserve the standard unit-constraints of the categories, we have
$(\id\tensor\epsilon\tensor\id)(\as) = \one$ as well.
Needless to say, we in addition verified all three identities with the help of computer algebra.

Altogether we have shown:

\begin{Prop}\label{prop:repS-repQ-monoid}
The natural isomorphism $\assocQ$ from~\eqref{eq:assoc-RepQ-via-Phi} with $\Phi$ as in \eqref{eq:Phi-via-f+-} defines an associator on $\repQ$. With respect to this associator, the equivalence $\funSQ: \repS \to \repQ$ with multiplicative structure $\isoG_{U,V}$ defined in~\eqref{eq:isoG} is monoidal.
\end{Prop}

\subsection{Quasi-Hopf structure on $\Q$: antipode and the $\Salpha$ and $\Sbeta$ elements}\label{eq:quasi-Hopf-al-be}

We can also introduce an antipode structure on $\Q$ that makes it a quasi-Hopf algebra.
The anti-automorphism $S$ is given by the same 
formulas~\eqref{eq:antipode} as for $\Q$.
The elements $\Salpha$ and $\Sbeta$
characterising the antipode can be found	by Proposition~\bref{prop:antipode}:
fixing $\Salpha=\one$ there is unique $\Sbeta$ satisfying all the axioms of  a quasi-Hopf algebra, 
namely
\begin{equation}
 \Salpha = \one\ ,\qquad \Sbeta = \idem_0 + \beta^2 \bigl(\K - 2\rmi\E\F\bigr)\idem_1 =  \idem_0 -2\rmi\beta^2\cas\idem_1 \ ,
\end{equation}
with  the  Casimir element $\cas$ defined under \eqref{eq:intro-al-be-def}.
These are central elements of $\Q$, and they are invertible
(since $\Sbeta^2=\one$).
We note that the element $\Sbeta$ is a linear combination of all the central primitive idempotents. Indeed, it can be written as
$\Sbeta = \idem_0 -\rmi\beta^2 (\idem_1^+ - \idem_1^-)$,
where idempotents $\idem_1^{\pm}$ are central primitive and correspond to the simple projective covers $\XX^{\pm}_2$ from~\bref{sec:repQ}:
\begin{equation}\label{idem-pm-1}
\idem_1^{\pm} = \bigl( \half\one \pm \cas\bigr)\idem_1\ .
\end{equation}

\subsection{Transporting the braiding}\label{sec:trans-braid}

The braiding on $\repQ$ is computed from that in $\catSF$ along the same lines as the associator. 
Consider the equivalence $\mathcal{P} := \funSQ\circ\funD:\catSF\to\repQ$. The functor $\mathcal{P}$ is monoidal via
\begin{equation}
\Pi_{U,V}: \mathcal{P}(U*V)\to \mathcal{P}(U) \tensor_{\repQ} \mathcal{P}(V)
\quad , \quad
\Pi_{U,V} =  \isoG_{\funD(U),\funD(V)}\circ\funSQ(\isoD_{U,V}) \ .
\end{equation}
Below we will write $\tensor$ instead of $\tensor_{\repQ}$ for brevity. The braiding $\sigma_{M,N}$ on $\repQ$ is 
uniquely 
determined by the braiding $c_{U,V}$ on $\catSF$ (see Section \bref{sec:SF-braiding}) by requiring commutativity of the diagram
\begin{equation}\label{eq:braiding-transport-comm-diag}
\xymatrix@R=22pt@C=42pt{
&\mathcal{P}(U*V)\ar[r]^{\mathcal{P}(c_{U,V})}\ar[d]^{\Pi_{U,V}}&\mathcal{P}(V*U)\ar[d]^{\Pi_{V,U}}&\\
&\mathcal{P}(U)\tensor\mathcal{P}(V)\ar[r]^{\brQ_{\mathcal{P}(U),\mathcal{P}(V)}}&\mathcal{P}(V)\tensor\mathcal{P}(U)&
}
\end{equation}
for all $U,V \in \catSF$. The resulting conditions can be evaluated sector by sector and are collected in Table \ref{tab:R-eqn-sectors}. We give the computation in the {\bf 10}-sector as an example.

\begin{table}[bt]
\begin{align*}
&{\bf00}~:~&&  \sflip_{U,V}\Big( \sflip_{S,S}(\delta_0) \,\hat.\, \gamma^{-1} \,\hat.\, (u \tensor v) \Big)
=
\sigma^{00}_{U,V}\Big( \delta_0 \,\hat.\, (u\tensor v)\Big)
\\[.6em]
&{\bf01}~:~&&  
\sflip_{U,\B \svtensor V}\Big(
\big[\sflip_{\B,U}\big(
\copS(a) 
 \,\hat.\, (\Bel \tensor \kappa)
  \,\hat.\, (\one \tensor u)\big)\big]
  \tensor v
\Big)
\\
&&&=
\sigma^{01}_{U,\B \svtensor V}\Big(
\big[\copS(a) 
 \,\hat.\, (\delta_1 \tensor \Bel)
  \,\hat.\,  (u \tensor \one) \big] \tensor v
\Big)
\\[.6em]
&{\bf10}~:~&&
\sflip_{\B \svtensor U, V} \circ (\id_\B \tensor \id_U \tensor \rho^V) 
\\
&&& \hspace{7em}
\circ (\id_\B \tensor \sflip_{S,U} \tensor \omega_V)  \Big(
\big[(\sflip_{S,S} \circ 
\copS(a) )
 \,\cdot\, (\Bel \tensor (\delta_1\kappa)) \big]
  \tensor u \tensor v
\Big)
\\
&&&=
\sigma^{10}_{\B \svtensor U,V}
\circ (\id_\B \tensor \id_U \tensor \rho^V) 
\circ (R_\Bel \tensor \sflip_{S,U} \tensor \id_V)  \Big(
\copS(a) \tensor u \tensor v
\Big)
\\[.6em]
&{\bf11}~:~&&  
\beta \cdot 
	\sflip_{\B \svtensor U, \B \svtensor V} 
\circ (\id_\B \tensor \sflip_{\B,U} \tensor \id_V) \Big(
\sflip_{\B,\B}\big[
(\id_\B \tensor L_\xi)\big( \copS(h\kappa^{-1}) 
 \,\cdot\, (\Bel \tensor \Bel)\big) \big]
  \tensor u \tensor v
\Big)
\\
&&&=
\sigma^{11}_{\B \svtensor U,\B \svtensor V}
\circ (\id_\B \tensor \id_U \tensor L_\xi \tensor \id_V) 
\\
&&& \hspace{7em}
\circ(\id_\B \tensor \sflip_{\B,U} \tensor \id_V)
\Big(
\big[\copS(h) \,\cdot\, (\Bel \tensor \Bel)\big] \tensor u \tensor v
\Big)
\end{align*}
\caption{Conditions determining $\sigma^{ij}$. 
The conditions have to hold for all $h \in \Salg_0$, $a \in \B$, $u \in U$, $v \in V$ and all $U \in \catSF_i$, $V \in \catSF_j$.
We have written $\,\hat.\,$ to denote the product in $\Salg \svtensor \Salg$ which includes a parity sign (though this only makes a difference in sector {\bf 00}). $L_\xi(u) := \xi.u$ denotes the left action with $\xi = \xp+\xm$,
and 
$\gamma^{-1} := \exp(-C) = (\one\tensor\one - \xm\tensor \xp + \xp\tensor \xm - \xp \xm\tensor \xp \xm)\cdot\idem_0 \tensor \idem_0$ 
and $\kappa := \exp(\frac12\hat C) = (\one - \xp \xm ) \idem_0$, 
see Section \bref{sec:assoc-SF}.
}\label{tab:R-eqn-sectors}
\end{table}

In computing $\Pi_{UV}$ note that $\isoG_{\funD(U),\funD(V)}$ is different from the identity only in sector {\bf 11}, see \eqref{eq:isoG}, and that $\funSQ(f) = f$ for all morphisms $f$ in $\repS$, see Section \bref{sec:funSQ-def}. To evaluate the above diagram
for $U\in\catSF_1$ and $V\in\catSF_0$, 
we thus only need to combine sectors {\bf 10} and {\bf 01} of $\Delta_{UV}$ as given in \eqref{eq:isoD10} and \eqref{eq:isoD01} (see also Figure \ref{fig:Delta_UV}) with the braiding of $\catSF$ as stated in 
Section~\bref{sec:SF-braiding}. In string diagram notation, the resulting condition is
\begin{equation}\label{eq:braid-10-example}
\raisebox{-0.5\height}{\includegraphics[scale=0.3]{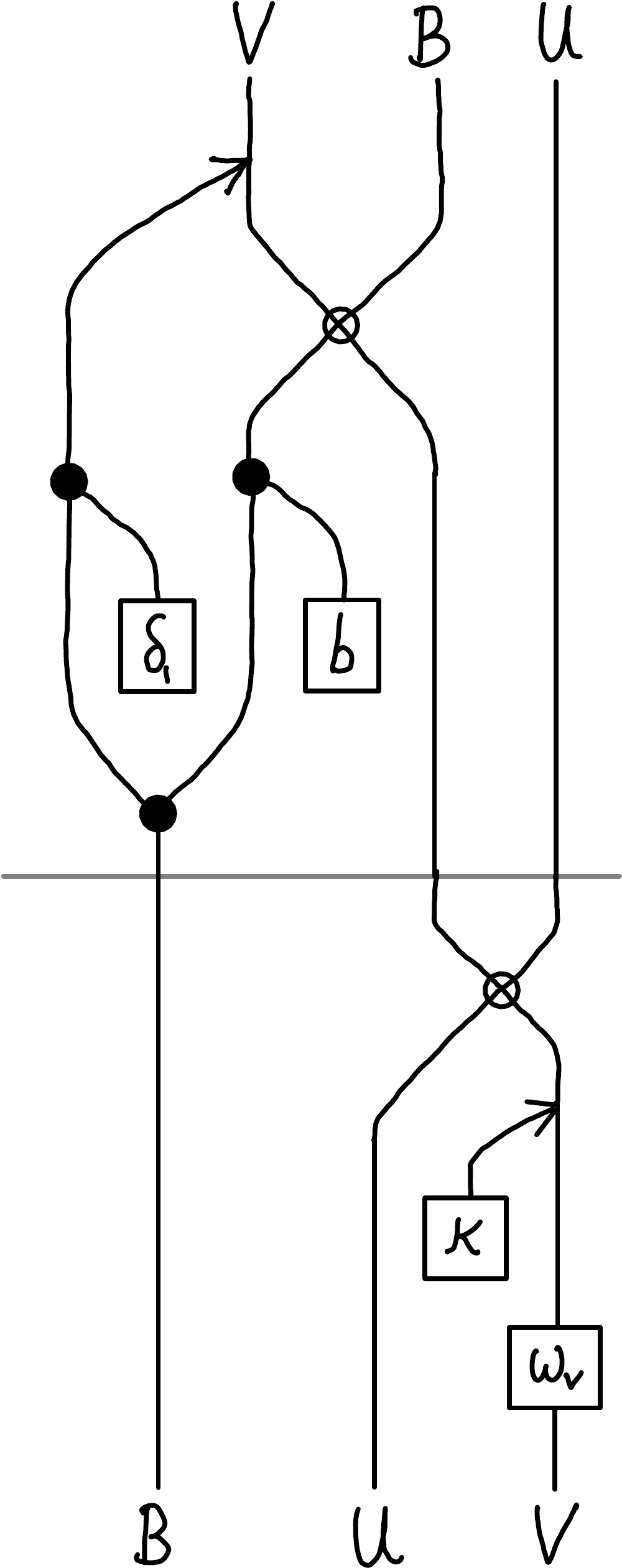}}
\qquad = \qquad
\raisebox{-0.5\height}{\includegraphics[scale=0.3]{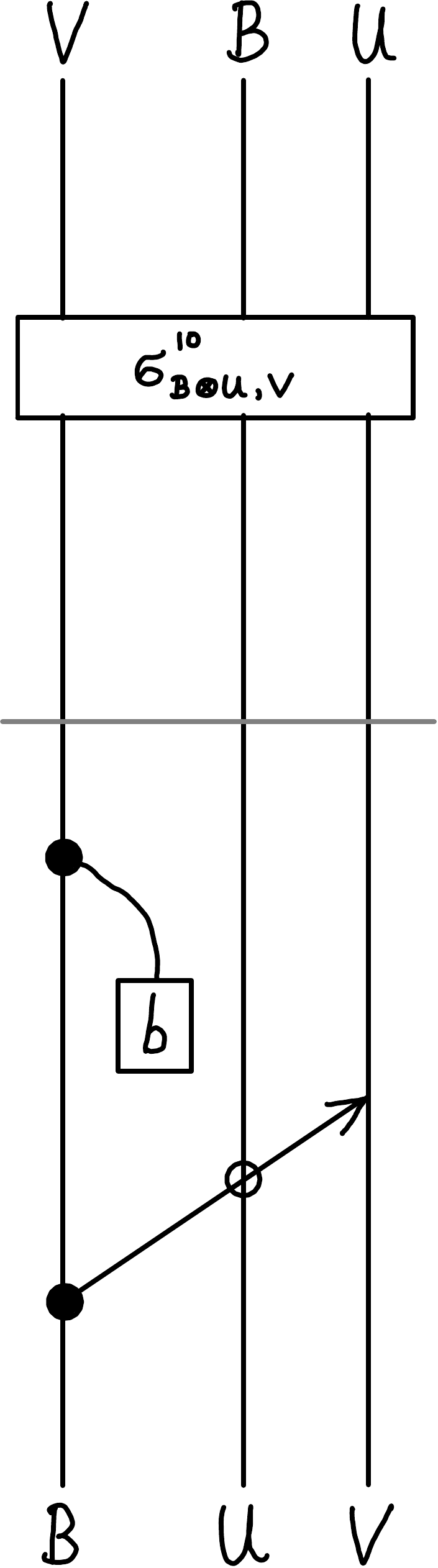}}
\qquad ,
\end{equation}
where 
$\kappa := \exp(\frac12\hat C)=(\one - \xp \xm ) \idem_0$.
The circle around the crossings has been added to stress that in these diagrams all braidings are in $\svect$, even though the diagram itself is a morphism in $\vect$ (or rather in $\repQ_1$).
The morphisms $\sigma^{ij}_{M,N}$ look simpler when written as $\sflip_{M,N} \circ (\text{action of modes})$. We therefore rewrite the left hand side of \eqref{eq:braid-10-example} as
\begin{equation}
\raisebox{-0.5\height}{\includegraphics[scale=0.3]{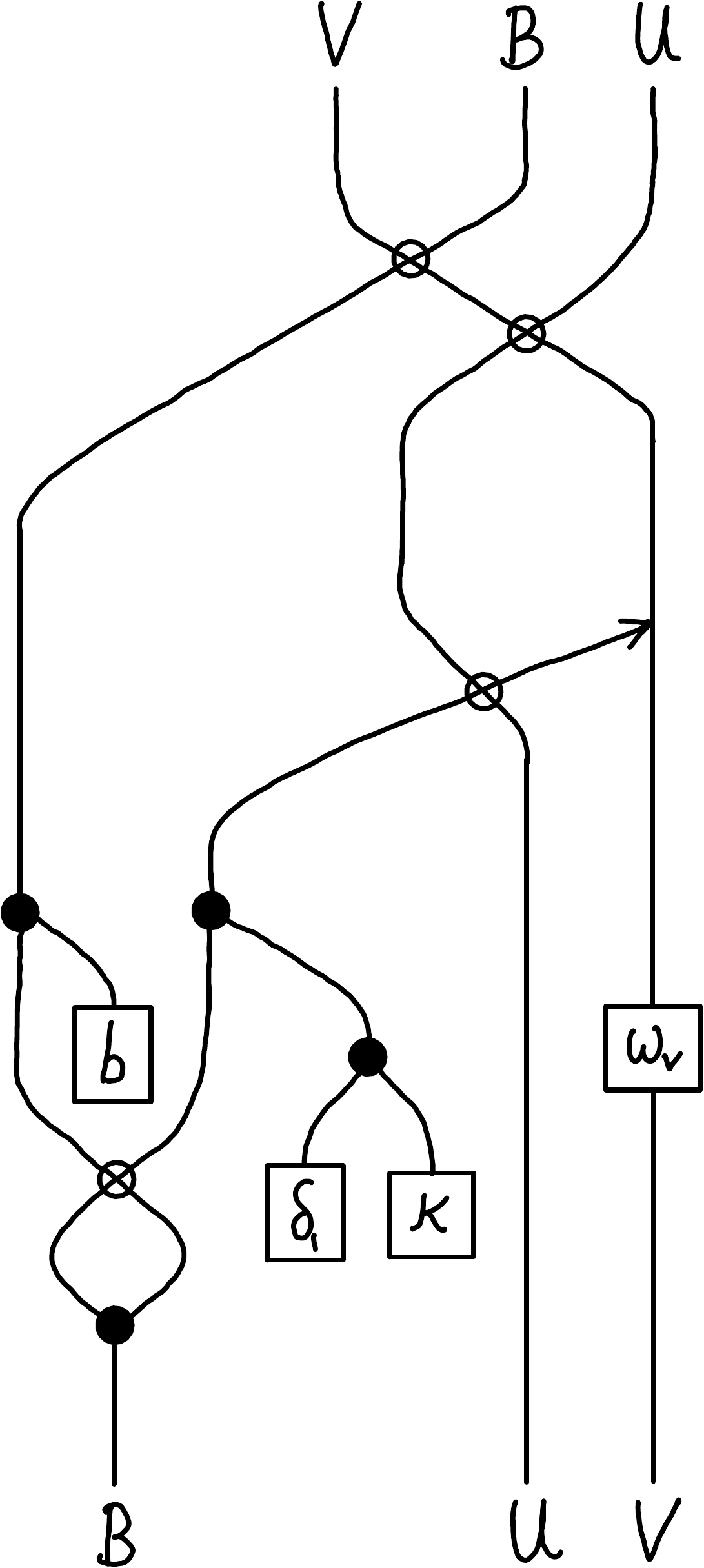}}
\qquad .
\end{equation}
This gives the formula for the {\bf 10}-sector listed in Table~\ref{tab:R-eqn-sectors}.

Since $\mathcal{P}$ is a monoidal equivalence, \eqref{eq:braiding-transport-comm-diag} is solved by a unique natural collection of isomorphisms $\{ \sigma_{M,N} : M \tensor_{\repQ} N \to N \tensor_{\repQ} M \}$. Hence the conditions in Table \ref{tab:R-eqn-sectors} have a unique solution. 
It is given by
\begin{align*}
\brQ^{00} &= \sflip \circ \bigl(\one\tensor\one - 2\ff^-\K\tensor\ff^+\bigr),\\
\brQ^{01} &= \sflip \circ \bigl(\one\tensor\one - (1+\rmi)\ff^-\K\tensor\ff^+
    -(1+\rmi)\ff^+\K\tensor\ff^-  +(1-\rmi)\ff^-\ff^+\tensor\one + 2\rmi\ff^-\ff^+\tensor\ff^-\ff^+\bigr),\\
\brQ^{10} &=  \sflip \circ\bigl(\one\tensor\one + (1 + \rmi) \ff^-\K\tensor\ff^+ + (1 + \rmi) \ff^+\K\tensor\ff^-  
\\ &
\hspace{10em}
   + (1 + \rmi) \one\tensor\ff^-\ff^+ 
- 2 \rmi \ff^-\ff^+\tensor\ff^-\ff^+\bigr)\cdot \one\tensor\K\ ,\\
\brQ^{11} &=   \sflip \circ  \ffrac{\beta}{\rmi} \bigl(\one\tensor\one - 2\rmi  \ff^-\K\tensor\ff^+
+ (\rmi-1)\one\tensor\ff^-\ff^+ 
\\ &
\hspace{10em}
- (1 + \rmi)\ff^-\ff^+\tensor\one +
 2  \ff^-\ff^+\tensor\ff^-\ff^+\bigr)\cdot \K\tensor\one\ .
\end{align*}
We verified (and found) this by computer algebra. 
	To do so, one has to remove the need to verify the conditions in Table \ref{tab:R-eqn-sectors} for all $U,V$. One uses that by naturality, $\sigma_{M,N}$ is uniquely determined by $\sigma_{\Q,\Q}$ for all $M,N \in \repQ$. Consider sector {\bf 10} as an example. It is sufficient to determine $\sigma^{10}_{\Q_1,\Q_0}$. We choose 
	$U = \Q_1$ (the underlying super-vector space), 
	$V = \Q_0$ (as a $\algGr$-module) and compose both sides of \eqref{eq:braid-10-example} with $\id_{\Q_0} \tensor \rho^{\Q_1}$, where $\rho^{\Q_1}$ denotes the action of $\B \subset \Salg_1$ on $\Q_1$. Then by naturality 
$(\id_{\Q_0} \tensor \rho^{\Q_1}) \circ \sigma^{10}_{\B \svtensor \Q_1,\Q_0} = \sigma^{10}_{\Q_1,\Q_0} \circ (\rho^{\Q_1} \tensor \id_{\Q_0})$. Since $\rho^{\Q_1}$ is surjective 
(for example, 
$\idem_1 = \Bel.\idem_1+\Cel.(\ff^-\idem_1)$, etc., with $\Cel = \xp.\Bel$), this determines $\sigma^{10}_{\Q_1,\Q_0}$ uniquely.

When expressed in terms of the generators $\E$, $\F$, and $\K$, the braiding $\sigma$ takes the form
\begin{align*}
\brQ^{00} &= \sflip \circ \bigl(\one\tensor\one +2\rmi\E\tensor\F\bigr)\ ,\\
\brQ^{01} &= \sflip \circ \bigl(\one\tensor\one - (1{-}\rmi)\E\tensor\F
    +(1{-}\rmi)\F\K\tensor\E\K  +(1{+}\rmi)\E\F\K\tensor\one + 2\rmi\E\F\K\tensor\E\F\K\bigr)\ ,\\
\brQ^{10} &=  \sflip \circ\bigl(\one\tensor\one + (1{-}\rmi) \E\tensor\F + (1{-}\rmi) \F\K\tensor\E\K     - (1{-}\rmi) \one\tensor\E\F\K -
 2 \rmi \E\F\K\tensor\E\F\K\bigr)\cdot \one\tensor\K \ ,\\
\brQ^{11} &=   \sflip \circ \ffrac{\beta}{\rmi}\bigl(\one\tensor\one - 2  \E\tensor\F
+ (1{+}\rmi)\one\tensor\E\F\K - (1{-}\rmi)\E\F\K\tensor\one -
 2  \E\F\K\tensor\E\F\K\bigr)\cdot\K\tensor\one \ .
\end{align*}
To finally recover the formula for $R$ as given in \eqref{RQ}, one still needs to solve 
$\sigma_{M,N}(m \otimes n) = \tau_{M,N}(R.(m \otimes n))$, 
where $\tau$ is the symmetric braiding in vector spaces, $\tau_{M,N}(m \otimes n) = n \otimes m$. To this end, we first observe that the braiding in super-vector spaces can be expressed as
\begin{equation}\label{sflip-om}
\sflip_{M,N} = \tau_{M,N} \circ \half(\id_M\tensor\id_N + \omega_M\tensor\id_N + \id_M\tensor\omega_N - \omega_M\tensor\omega_N) \ ,
\end{equation}
where $\omega_M(m) =  (\idem_0-\rmi\idem_1)\K.m$ as in \eqref{Z2-repQ}. This produces the prefactor composed of the generators $\one$ and $\K$ in \eqref{RQ}. Since this R-matrix arises as a transported braiding from a braided monoidal category, by construction it
gives a morphism in $\repQ$ (i.e.\ it satisfies \eqref{eq:RD=DopR}), and it obeys the two hexagon identities (i.e.\ it satisfies the two identities in \eqref{eq:R-mat-hex12}).

Together with Proposition \bref{prop:repS-repQ-monoid} and Section \bref{eq:quasi-Hopf-al-be} we have now proved Theorem \bref{thm:2}. 
In fact, along the way we have also proved Theorem \bref{thm:3}: by Propositions \bref{prop:catSF-repS-monoid} and \bref{prop:repS-repQ-monoid}, the equivalence $\mathcal{P}$ is monoidal. By construction of the R-matrix of $\Q$, the equivalence $\mathcal{P}$ is also braided.

\begin{Rem}\label{rem:trans-braid-0}
It is known that the quotient of $\Q$ by the ideal generated by $(\one-\K^2)$ (the algebra isomorphic to $\Q_0$) is a quasi-triangular
	Hopf algebra (rather than quasi-Hopf)
with the standard 
 R-matrix (see, e.g.,~\cite{Kassel})
\begin{align}\label{eq:R00}
\Rst &= \half\sum_{m,n=0,1}(-1)^{mn}\K^m\tensor\K^n\bigl(\one\tensor\one + 2\rmi \E\tensor\F\bigr)\\
&= \half\sum_{m,n=0,1}(-1)^{mn}\K^m\tensor\K^n\bigl(\one\tensor\one - 2\ff^-\K\tensor\ff^+\bigr) \ .
\nonumber
\end{align}
This coincides with the $00$-component of the R-matrix~\eqref{RQ} just computed.
Introducing the quasi-Hopf structure on $\Q$ given by the coassociator $\as$ and the antipode elements $\Salpha$ and $\Sbeta$, 
we have thus extended the quasi-triangular structure from the quotient $\Q_0$ onto the whole quantum group $\Q$. 
\end{Rem}

\subsection{Transporting the ribbon twist}

The ribbon twist in $\catSF$ is given in Section \bref{eq:SF-twist}. The ribbon twist in $\repQ$ is uniquely determined by 
\begin{equation}\label{eq:theta-transp}
	\theta_{\mathcal{P}(U)} = \mathcal{P}(\theta_U) : \mathcal{P}(U) \longrightarrow \mathcal{P}(U) \ .
\end{equation}
In sector {\bf 0} this means that for $M \in \repQ_0$, $m \in M$, we have $\theta_M(m) = (1+2 \ff^+ \ff^-).m$. 

In sector {\bf 1} the calculation is more interesting. Condition \eqref{eq:theta-transp} now reads $\theta_{\B \svtensor V} = \beta^{-1} \cdot \id_\B \tensor \omega_V$. To proceed, we note that, for all $a \in \B$,
\begin{equation}
	\omega_\B(a) = (\ff^-\ff^+ - \ff^+\ff^-).a \ .
\end{equation}
We can therefore write, for $a \in \B$ and $u \in U$, 
\begin{equation}
\theta_{\B \svtensor U}(a \tensor u) = \beta^{-1} \cdot \omega_{B \svtensor U}\Big(
\big[(\ff^-\ff^+ - \ff^+\ff^-)a\big] \tensor u \Big) \ ,
\end{equation}
where we used that $\omega$ is monoidal. Since in sector {\bf 1}, $\omega$ is given by $-\rmi\K$, see \eqref{Z2-repQ},  we have, for $M \in \repQ_1$, $m \in M$, that $\theta_M(m) = \beta^{-1} (-\rmi) \K(\one - 2 \ff^+\ff^-).m$. Altogether, 
\begin{equation}\label{eq:theta-2sectors}
\theta_M(m) = \Big( \idem_0 (\one + 2 \ff^+\ff^-) - \rmi \beta^{-1} \idem_1 \K (\one-2 \ff^+ \ff^-) \Big).m \ ,
\end{equation}
where now $M \in \repQ$ and $m \in M$.

In our convention (and in that of, e.g.,\ \cite{Kassel}), acting with 
the ribbon element $\ribbon$ of a Hopf algebra gives the inverse twist. Taking the inverse of \eqref{eq:theta-2sectors} produces
\begin{equation}\label{eq:ribbon-el-Q}
	\ribbon = (\idem_0 - \rmi \beta \K \idem_1)(\one - 2 \ff^+ \ff^-) \ .
\end{equation}
By construction, $\ribbon$ is central (as its left-action is an intertwiner) and invertible (since the ribbon twist in $\catSF$ is).
Its decomposition on the three primitive central idempotents
 $\idem_0=\half(\one+\K^2)$ and $\idem_1^{\pm}$ defined in~\eqref{idem-pm-1}, 
 and central nilpotents
	$\nilp^{\pm} = \half \E\F (\one\pm\K)\idem_0$
is
\begin{equation}\label{ribbon-decomp}
\ribbon = \idem_0 + \beta(\idem_1^+ - \idem_1^-)
+2\rmi(\nilp^+-\nilp^-) \ .
\end{equation}

Let $M = R_{21}R$ be the monodromy matrix. Explicitly, it is given by, for $\q=\rmi$,
\begin{equation}\label{bar-M-q=i}
   M
  =\ffrac{1}{4}
  \sum_{m,n=0}^{1}
  \sum_{i,j=0}^{3}
  \big(\tfrac{\beta^2}{\q}\big)^{ij + m(i + j + 1)}
  (\q - \q^{-1})^{m + n}\,
  \q^{- m^2 - m j + 2n j - 2n i - i j + m i} \,
  \F^{m} \E^{n} \K^{j}\tensor \E^{m} \F^{n} \K^{i} \ .
\end{equation}
It is straightforward to see that the conditions $M\Delta(\ribbon) = \ribbon\tensor\ribbon$ and $S(\ribbon)=\ribbon$
 hold (see Definition \bref{def:ribbonel}), and so we get:

\begin{Lemma}
$\ribbon$ 
from~\eqref{ribbon-decomp}
 is a ribbon element for the quasi-triangular quasi-Hopf algebra $\Q$.
\end{Lemma}

\begin{Rem}
The Hopf algebra $\UresSL2$ at $\q = e^{\rmi\pi/p}$ can be realised as a Hopf subalgebra of a quasi-triangular Hopf algebra $\overline D$ of twice the dimension of $\UresSL2$, see \cite[Sect.\,4.1]{[FGST]}. It turns out that the monodromy matrix and ribbon element of $\overline D$ lie in the subalgebra $\UresSL2 \tensor \UresSL2 \subset \overline D \tensor \overline D$. For $p=2$, these expressions agree with \eqref{eq:ribbon-el-Q} and \eqref{bar-M-q=i} in the case $\beta = \exp(+\pi \rmi / 4)$,
see \cite[Sect.\,4.2\,\&\,4.6]{[FGST]}. It is verified in \cite{[FGST]} that this monodromy matrix and ribbon element reproduce the $SL(2,\oZ)$-action on the  $(3p-1)$-dimensional space of $\Walg_p$-torus amplitudes. 
In the present paper, the symplectic fermion case is $\beta = \exp(-\pi \rmi / 4)$
(the difference to \cite{[FGST]} arises from the convention of how to define the $T$-action in terms of $\ribbon$ and from the normalisation convention for the (co)integral, see Appendix \bref{app:SL2-defs}).
We show in Appendix~\bref{app:SL2Z} that our monodromy matrix and the ribbon element  at 
any $\beta$ with $\beta^4=-1$
do define  an $SL(2,\oZ)$-action on the centre of the quasi-Hopf algebra $(\Q,\as)$, and that 
at $\beta = \exp(-\pi \rmi / 4)$ 
this  representation of $SL(2,\oZ)$ is isomorphic
 to the one on symplectic fermion torus blocks
and to the one in~\cite{[FGST]}.
\end{Rem}

\appendix

\section{Conventions for quasi-bialgebras and quasi-Hopf algebras}\label{app:qausi-Hopf}
In this appendix, we review basics of theory of quasi-Hopf 
algebras~\cite{Dr-quasi}
(for conventions, we follow~\cite[Sec. 16.1]{ChPr}).
In this paper (as in \cite[Sec. 16.1]{ChPr}) we make the
\begin{quote}
	\textit{Assumption 1:} We will only consider quasi-Hopf algebras $A$ such
	 that the unit isomorphisms $\lambda_U$ and $\rho_U$  in $\rep A$ are as in $\vect$.
\end{quote}	 
This simplifies for example the $\epsilon$-conditions~\eqref{eq:eps-Delta} and~\eqref{eq:counital} below as they do not involve non-trivial invertible elements $l$ and $r$.

\begin{Dfn} \label{def:quasi-bialgebra} A quasi-bialgebra (say, over $\oC$) is an associative algebra $A$  over $\oC$ together with
algebra homomorphisms: the counit $\epsilon: A\to \oC$ and the comultiplication $\Delta: A\to A\tensor A$, and an
invertible element $\as\in A\tensor A\tensor A$ called the coassociator, satisfying the following conditions:
\begin{equation}\label{eq:eps-Delta}
(\epsilon\tensor\id)\circ\Delta = \id =  (\id\tensor\epsilon)\circ\Delta,
\end{equation}
and
\begin{equation}\label{eq:as-intertwiner}
(\Delta\tensor\id)(\Delta(a)) = \as \, \big( (\id\tensor\Delta)(\Delta(a))\big) \, \as^{-1},
\end{equation}
for all $a\in A$; and the coassociator $\as$ is counital
\begin{equation}\label{eq:counital}
(\id\tensor\epsilon\tensor\id)(\as) = \one\tensor\one
\end{equation}
and is a $3$-cocycle
\begin{equation}\label{eq:3-cocycle}
(\Delta\tensor\id\tensor\id)(\as)\cdot (\id\tensor\id\tensor\Delta)(\as) = (\as\tensor\one)\cdot(\id\tensor\Delta\tensor\id)(\as)\cdot(\one\tensor\as).
\end{equation}
\end{Dfn}

The associativity isomorphisms $\alpha^{\rep A}_{U,V,W} : U \tensor (V \tensor W) \to (U \tensor V) \tensor W$ for the tensor product of $\rep A$ are related to the coassociator $\as$ of $A$ by
\begin{equation*}
\alpha^{\rep A}_{U,V,W}(u\tensor v\tensor w) = \as.(u\tensor v\tensor w)\ ,
\end{equation*}
for any elements $u$, $v$, $w$ in  $A$-modules $U$, $V$, and $W$, respectively.
\begin{Dfn}\label{def:quasi-triang_for_quasi-hopf}
	 A quasi-triangular quasi-bialgebra is a quasi-bialgebra  $A$
		equipped with an invertible element $R\in A\tensor A$ called \textit{the universal R-matrix} 
	(or \textit{R-matrix} for short)
		such that
\begin{equation}\label{eq:RD=DopR}
R\Delta(a) = \Delta^{\mathrm{op}}(a)R
\end{equation}
for  all $a\in A$;
and the quasi-triangularity conditions
\begin{equation}\label{eq:R-mat-hex12}
\begin{split}
\Delta\tensor\id(R) &= \as_{231}^{-1} R_{13} \as_{132} R_{23} \as^{-1},
\\
\id\tensor\Delta(R) &= \as_{312} R_{13} \as_{213}^{-1} R_{12} \as.
\end{split}
\end{equation}
Here, we set $\as_{231}= \sum_{(\as)}\as''\tensor \as'''\tensor \as'$, \textit{etc.}, for an expansion $\as= \sum_{(\as)}\as'\tensor \as''\tensor \as'''\in A\tensor A\tensor A$, and also $R_{13}=\sum_{(R)}R_1\tensor\one\tensor R_2$, for an expansion $R=\sum_{(R)}R_1\tensor R_2$. 
\end{Dfn}

The  braiding isomorphisms $\sigma_{U,V}$ in  $\rep A$ are related to the  
 R-matrix by
$$\sigma_{U,V}(u \otimes v) = \tau_{U,V}(R. (u \otimes v)) \ ,$$
 where $\tau$ is the symmetric braiding in vector spaces, $\tau_{U,V}(u \otimes v) = v \otimes u$. Due to~\eqref{eq:R-mat-hex12},  the isomorphisms satisfy 
	 the hexagon axioms of a braided monoidal category.
Applying the linear map $\id\tensor\epsilon\tensor\id$ to both equations in~\eqref{eq:R-mat-hex12} and using the counital condition~\eqref{eq:counital}, we obtain the following proposition.

\begin{Prop}\label{prop:eps-R}
\cite[Sec. 3]{Dr-quasi}
 Under Ass.~1, for a quasi-triangular quasi-bialgebra we have two relations
\begin{equation}
(\epsilon\tensor \id)(R) = \one = (\id\tensor \epsilon)(R).
\end{equation}
\end{Prop}

Proposition~\bref{prop:eps-R}  corresponds to the commutativity of the diagram involving the left- and right-units and the braiding.
	Altogether, we have now turned
$\rep A$ into a braided monoidal category.

\begin{Dfn} \label{def:quasi-Hopf_ass1}
	 Under Ass.~1, a \textit{(quasi-triangular) quasi-Hopf algebra} is a 
(quasi-triangular) quasi-bialgebra  $A$ equipped with an anti-homomorphism $S: A\to A$ called
	 \textit{the antipode}, and elements $\Salpha$, $\Sbeta\in A$, such that
\begin{equation}\label{eq:Salpha-1}
\sum_{(a)}S(a')\Salpha a'' = \epsilon(a)\Salpha\ ,\qquad
\sum_{(a)}a'\Sbeta S(a'') = \epsilon(a)\Sbeta
\end{equation}
for all $a\in A$; and
\begin{equation}\label{eq:Salpha-2}
\sum_{(\as)}S(\as')\Salpha \as''\Sbeta  S(\as''') = \one\ ,\qquad
\sum_{(\as^{-1})}(\as^{-1})'\Sbeta S((\as^{-1})'')\Salpha  (\as^{-1})''' = \one \ .
\end{equation}
\end{Dfn}

\begin{Prop}\cite[Prop. 1.1]{Dr-quasi}\;
If the triple $\tilde{S}$, $\tilde{\Salpha}$, $\tilde{\Sbeta}$ gives another antipode structure in $A$ then
there exists unique element $U\in A$ such that
 \begin{equation}
\tilde{S}(a) = U S(a)U^{-1}\ ,\qquad
\tilde{\Salpha} = U \Salpha\ ,\qquad
\tilde{\Sbeta} = \Sbeta U^{-1} \ .
\end{equation}
\end{Prop}

So, $S$, $\Salpha$ and $\Sbeta$ are uniquely determined up to the conjugation by a unique element $U$.

\newcommand{\ias}{\bar{\as}}

\begin{Prop}\label{prop:antipode}
If $A$ is a Hopf algebra and $\as\in A^{\tensor3}$ is an invertible counital element satisfying~\eqref{eq:as-intertwiner} and the $3$-cocycle condition~\eqref{eq:3-cocycle} then 
\begin{enumerate}
\item the element
\begin{equation}\label{eq:Sbeta-Psi}
\Sgamma\equiv\sum_{(\as)}S(\as') \as''  S(\as''')=\sum_{(\as^{-1})}(\as^{-1})' S((\as^{-1})'')  (\as^{-1})'''
\end{equation}
is central.
\item assuming that $\Sgamma$ has an inverse,
$(A,\as)$ is a quasi-Hopf algebra with the same antipode $S$ and $\Salpha=\one$ and $\Sbeta=\Sgamma^{-1}$.
\end{enumerate}
\end{Prop}
\begin{proof}
To prove the  second equality in~\eqref{eq:Sbeta-Psi} we use the
condition~\eqref{eq:3-cocycle}  in the following form
\begin{equation}\label{eq:3-cocycle-2}
 (\as^{-1}\tensor\one)\cdot(\Delta\tensor\id\tensor\id)(\as)\cdot (\id\tensor\id\tensor\Delta)(\as) =(\id\tensor\Delta\tensor\id)(\as)\cdot(\one\tensor\as).
\end{equation}
and apply on both sides the linear map 
$\psi=\mu^A\circ (\mu^A\tensor\mu^A)\circ(\id\tensor S\tensor \id\tensor S)$, where $\mu^A$ is the multiplication in $A$. Computing the image of the map $\psi$, we use the properties
$$\psi\bigl(t\cdot (\Delta(a)\tensor\one\tensor\one)\bigr)=\psi\bigl(t\cdot (\one\tensor\one\tensor\Delta(a))\bigr)=
\psi\bigl((\one\tensor\Delta(a)\tensor\one)\cdot t\bigr)=\epsilon(a)\psi(t)\ ,$$
for any $t\in A^{\tensor4}$ and $a\in A$, together with the counital properties of $\as$. The right-hand side of~\eqref{eq:3-cocycle-2} under $\psi$ is then $\sum_{(\as)}S(\as') \as''  S(\as''')$,  and $\sum_{(\as^{-1})}(\as^{-1})' S((\as^{-1})'')  (\as^{-1})'''$ is for  the left-hand side. We thus see that~\eqref{eq:Sbeta-Psi} is true. To prove 
 that $\Sgamma$ is central, we first note the identities
$a=\sum_{(a)}a'\epsilon(a'') = \sum_{(a)} a' S(a'') a'''$,
 for any $a\in A$, because $A$ is a Hopf algebra,
and apply them for the product $a\Sgamma$ (with the short-hand notation $\ias=\as^{-1}$):
\begin{equation*}
a\Sgamma = \sum_{(a)(\ias)}a'\ias' S(\ias'') \epsilon(a'') \ias''' = \sum_{(a)(\ias)}a'\ias' S(\ias'')  S(a'') a''' \ias''' = 
 \sum_{(a)(\ias)}a'\ias' S(a'' \ias'') a''' \ias'''\ .
\end{equation*}
We apply the linear map $\id\tensor S\tensor \id$ on~\eqref{eq:as-intertwiner} multiplied by $\as^{-1}$ on the left, which results in
\begin{equation*}
\sum_{(a)(\ias)} a' \ias'\tensor S(a'' \ias'')\tensor a'''\ias''' = \sum_{(a)(\ias)}  \ias' a'\tensor S(\ias'' a'')\tensor \ias''' a'''\ .
\end{equation*}
Then, we can continue with rewriting $a\Sgamma$:
\begin{equation*}
a\Sgamma =  \sum_{(a)(\ias)}\ias' a' S(\ias'' a'') \ias''' a''' =  \sum_{(a)(\ias)}\ias' \epsilon(a') S(\ias'') \ias''' a'' = \Sgamma a\ ,
\end{equation*}
i.e., $\Sgamma$ is central indeed. 

 For the second point, $(A,\as)$ is a quasi-bialgebra by Definition~\bref{def:quasi-bialgebra}.
Using the centrality of $\Sgamma$ and the assumption that it has the inverse, the conditions~\eqref{eq:Salpha-1} are satisfied with $\Salpha=\one$ and $\Sbeta=\Sgamma^{-1}$.
Using~\eqref{eq:Sbeta-Psi}, $\Sbeta$ and $\Salpha$ satisfy~\eqref{eq:Salpha-2}, and thus $(A,\as)$ is the quasi-Hopf algebra.
\end{proof}

 A quasi-triangular quasi-Hopf algebra $A$ is called \textit{ribbon}  if
it contains a ribbon element $\ribbon$ defined in the same way as for ordinary Hopf algebras, see~\cite{Yorck}:

\begin{Dfn}\label{def:ribbonel}
A nonzero central element $\ribbon\in A$ is called a ribbon element if it
satisfies
\begin{equation}\label{def-ribbon}
  \Delta(\ribbon)=M^{-1}(\ribbon\tensor\ribbon),\qquad
 S(\ribbon)=\ribbon.
\end{equation}
\end{Dfn}

In a ribbon quasi-Hopf algebra $A$, we have the identities~\cite{[AC],Yorck}
\begin{equation}\label{u-ribbon}
   \ribbon^2= \sqs S(\sqs),\quad \epsilon(\ribbon)=1,
\end{equation}
where $\sqs$ is the (generalisation of the) canonical Drinfeld element defined as
\begin{equation}\label{sqs}
\sqs = \sum_{(\as),(R)} S(\as''\Sbeta S(\as'''))S(R'')\Salpha R'\as'
\end{equation}
and it satisfies $S^2(a) = \sqs a\sqs^{-1}$, for any $a\in A$.
The action by $\sqs$ is a canonical intertwiner between any $A$-module $U$ and its double dual $U^{**}$. 
Recall that the (left) dual $U^*$ for $U$ in $\rep A$ is defined as the vector space of $\oC$-linear maps $U\to\oC$ and the left $A$-action on $U^*$ is
\begin{equation}\label{eq:dualU}
a\cdot f(u) = f(S(a)u),\qquad u\in U,\quad f\in U^*, \quad a\in A.
\end{equation}
This is as in the case of Hopf algebras.

\section{$SL(2,\oZ)$-action on the centre of the quasi-Hopf algebra $(\Q,\as)$}\label{app:SL2Z} 
In this section, we first
recall the standard $SL(2,\oZ)$-action~\cite{[LM],[Lyu]} 
for a factorisable Hopf algebra, following conventions in~\cite{[FGST]}, and 
reformulate it for the centre 
$\cZ\equiv\cZ(\Q)$ 
of our quasi-Hopf algebra~$(\Q,\as)$.  Its definition
involves the ribbon element and the Drinfeld and Radford mappings.
 We then establish for $\beta=e^{\pm\rmi\pi/4}$
the equivalence to the $SL(2,\oZ)$-representation obtained in~\cite{[FGST]}.

\subsection{Notations and general definitions}\label{app:SL2-defs}
We define the representation $\repLy$ of $\SLiiZ$ on the centre $\cZ$ of $\Q$
as follows:
the operators
$\modS\equiv\repLy(S):\cZ\to\cZ$ and $\modT\equiv\repLy(T):\cZ\to\cZ$
are
\begin{equation}\label{TS-def}
  \modS(a) =
  \radmap\bigl(\drmap^{-1}(a)\bigr),
  \quad
  \modT(a)=b\,\modS^{-1}\bigl(\ribbon\bigl(\modS(a)\bigr)\bigr),
  \qquad a\in\cZ,
\end{equation}
where $\ribbon$ is the ribbon element, $\drmap$ is the Drinfeld
mapping, $\radmap$ is the Radford mapping, and $b$ is a
normalisation factor which will be fixed later as
\begin{equation}\label{eq:projectivity-factor}
	b = \beta^2 e^{2 \pi \rmi /3} \ .
\end{equation}
Note that in the symplectic fermion case, $\beta = e^{- \pi \rmi/4}$ and so $b = e^{-2\rmi\pi c/24}$ with $c=-2$.

We recall now the definition of the main ingredients, the  Drinfeld and Radford mappings
for quasi-triangular Hopf algebras,
see also~\cite[App.\,A]{[FGST]}. 
Given the $M$-matrix for $A$,  i.e., $M= R_{21} R\in A\tensor A$, \textit{the Drinfeld mapping}
$\drmap:A^*\to A$ is defined as
\begin{equation}\label{drinfeld-def}
  \drmap(\varphi)=(\varphi\tensor\id)M \ .
\end{equation}
A quasi-triangular (quasi-)Hopf algebra is  called \textit{factorisable}  if the map $\drmap$ is surjective.
It is well known~\cite{[Drinfeld]} that
  in a factorisable Hopf algebra $A$, the Drinfeld mapping
  $\drmap:A^*\to A$ intertwines the adjoint and coadjoint actions of
  $A$ and its restriction to the space $\Ch(A)$ of $q$-characters gives
  an isomorphism of associative algebras $\Ch(A)\xrightarrow{\sim}\cZ(A)$,
  where $\cZ(A)$ is the space of the adjoint-action ($\ad_a(x)=\sum_{(a)} a'xS(a'')$) invariants, the centre of $A$, while $\Ch(A)$ is by definition 
  the space of invariants in $A^*$ with respect to the coadjoint action of $A$, or equivalently
\begin{equation*}
  \Ch(A)
  = \bigl\{\varphi\in A^* \bigm| \varphi(xy)=\varphi\bigl(S^2(y)x\bigr)
  \quad \forall x,y\in A\bigr\}.
\end{equation*}

For a Hopf algebra $A$, a \textit{right integral}~$\rint$ is a linear
functional on $A$ satisfying
\begin{equation*}
  (\rint\tensor\id)\Delta(x)=\rint(x)\one
\end{equation*}
for all $x\,{\in}\, A$.  
Whenever such a functional exists, it is unique up to multiplication with a nonzero constant.
For a factorisable Hopf algebra, the integral can be  normalised~\cite[Sec.~3.8]{[Lyu]} (up to a sign) by requiring 
\begin{equation}\label{mu-mu}
(\rint\tensor\rint)(M)=1\ .
\end{equation}
The left--right \textit{cointegral}~$\coint$ 
is an element in $A$ such
that
\begin{equation*}
  x\coint=\coint x =\epsilon(x)\coint\ ,\quad\forall x\in A\ .
\end{equation*}
	We normalise the cointegral by requiring $\rint(\coint)=1$.
Let $A$ be a
Hopf algebra with right integral $\rint$ and left--right
cointegral~$\coint$.  \textit{The Radford mapping} $\radmap:A^*\to A$ and its
inverse $\radmap{}^{-1}:A\to A^*$ are given by
\begin{equation}\label{radford-def}
  \radmap(\varphi)
  =\sum_{(\coint)}\varphi(\coint')\coint''\ ,
  \quad
  \radmap{}^{-1}(x)=\rint(S(x)\,-)\ ,
\end{equation}
where `$-$' stands for an argument from $A$. The map $\radmap$ has the important property that it intertwines the coregular and regular actions of $A$ on $A^*$ and $A$, respectively.

Below we will apply these expressions for the Drinfeld and Radford mapping to our quasi-Hopf algebra $(\Q,\as)$.
This is motivated by the fact that the definition of adjoint and  regular representations is the same, and so their duals are also the same, see the note around~\eqref{eq:dualU} 
	(and of course by the outcome that we do indeed get an $\SLiiZ$-action on $\cZ(\Q)$).
The main difference to the Hopf-algebra case appears in the definition of the balancing element and so in a special basis of $q$-characters, which we discuss now.

In order to compute the $\SLiiZ$-action~\eqref{TS-def} explicitly, we need a basis in the space~$\Ch(A)$ of $q$-characters. 
In a ribbon quasi-Hopf algebra, we define the balancing element as~\cite{[AC]}
\begin{equation}\label{balance-ribbon}
  \balance=\Sbeta S(\Salpha)\ribbon^{-1}\sqs \ ,
\end{equation}
with the canonical Drinfeld element $\sqs$ defined in~\eqref{sqs}.
The balancing element $\balance$ is group-like and allows
constructing the ``canonical'' $q$-character of an $A$-module~$V$:
\begin{equation}\label{qCh}
\qtr_{V} \equiv\mathrm{Tr}_{V}(\balance^{-1}\,-)
  \in\Ch(A) \ .
\end{equation}
The map $\qtr: V\mapsto \qtr_{V} $ defines then a homomorphism of the Grothendieck ring of $\rep A$ to the ring of
  $q$-characters.
 
\subsection{The $\SLiiZ$-action on $\cZ$ of $(\Q,\as)$}
For our quasi-Hopf algebra $(\Q,\as)$, with the $M$-matrix in~\eqref{bar-M-q=i}, the normalised right integral $\rint$ and the left-right cointegral $\coint$ are
\begin{equation}\label{int-coint}
\rint(\E^m\F^n\K^l) = \ffrac{\beta^2}{\rmi} \delta_{m,1} \delta_{n,1}\delta_{l,3} 
\qquad , \qquad
\coint = \ffrac{\beta^2}{\rmi} \E\F\sum_{j=0}^3 \K^j \ ,
\end{equation}
where we assume our usual condition $\beta^4=-1$, so the coefficient in~\eqref{int-coint} is just a sign 
(which is not fixed by~\eqref{mu-mu} and the choice $\frac{\beta^2}{\rmi}$ is our convention).

Using~\eqref{balance-ribbon}, we compute the balancing element
\begin{equation}
\balance = \K^{\pm1}\ ,\qquad \text{for}\quad \beta^2=\mp\rmi\ .
\end{equation}
We now compute the image of the Grothendieck ring of $\rep\Q$ in the centre $\cZ$ using the composition $\drmap\circ\qtr$, see~\eqref{drinfeld-def} and~\eqref{qCh}:
  \begin{equation}
  \XX^{\pm}_s \mapsto \cchi^{\pm}_s\equiv \drmap\bigl(\qtr_{\XX^{\pm}_s}\bigr) = \sum_{(M)}\mathrm{Tr}_{\XX^{\pm}_s}(\balance^{-1}M')M''\ ,\qquad s=1,2\ ,
  \end{equation}
  using the $M$-matrix in~\eqref{bar-M-q=i}. For $s=1$, the image does not depend on $\beta$:
  \begin{equation}
  \cchi^+_1=\one\ ,\qquad \cchi^-_1=-\K^2\ ,
  \end{equation}
  while for $s=2$ the image depends on $\beta^2$:
  \begin{equation}
  \begin{split}
  \cchi^+_2=4\cas\ ,\qquad \cchi^-_2=-4\cas\K^2\ ,&\qquad \text{for}\quad \beta^2=\rmi\ ,\\
  \cchi^+_2=-4\cas\K^2\ ,\qquad \cchi^-_2=4\cas\ ,&\qquad  \text{for}\quad \beta^2=-\rmi\ ,
  \end{split}
  \end{equation}
where $\cas$ is the Casimir element defined under \eqref{eq:intro-al-be-def}.
This result agrees with the one in~\cite{[FGST]} corresponding to $\beta=e^{\pi\rmi/4}$.
Recall that the centre  $\cZ$ is $5$-dimensional~\cite{[FGST]} and is spanned by the  idempotents $\idem_0$, $\idem_1^{\pm}$ 
defined in~\eqref{idem-01} and~\eqref{idem-pm-1}, and the two nilpotents $\nilp^{\pm}$ defined just before~\eqref{ribbon-decomp}. 
Using the functionals  $\qtr_{\XX^{\pm}_s}$, we have found the four basis elements as images of $\drmap$, while to construct the fifth one could use the pseudo-trace from~\cite[Sec. 3.2]{[GT]}, a $q$-character associated to the projective module $\PP^+_1\oplus \PP^-_1$.
The composition $\radmap\circ\qtr$ is evaluated as
  \begin{equation}
  \XX^{\pm}_s \mapsto \radmap^{\pm}_s\equiv \radmap\bigl(\qtr_{\XX^{\pm}_s}\bigr) = \sum_{(\coint)}\mathrm{Tr}_{\XX^{\pm}_s}(\balance^{-1}\coint')\coint''\ ,\qquad s=1,2\ ,
  \end{equation}
  with the cointegral in~\eqref{int-coint}. The images have the following  dependence  on $\beta$:
  \begin{equation}
  \radmap^{\pm}_1=\ffrac{4\beta^2}{\rmi}\nilp^{\pm}\ , \qquad
  \radmap^{\pm}_2=\pm 4\idem^{\pm}_1\ .
  \end{equation}

Now, we are ready to formulate the main result of this section.

\begin{Prop}\label{prop:sl2z-action}
The $\SLiiZ$-action~\eqref{TS-def} on the centre $\cZ$ of the quasi-Hopf algebra  $(\Q,\as)$ has the decomposition
\begin{equation}\label{Z-decomp}
\cZ|_{\SLiiZ} = \oC^2 \oplus \oC^3\ .
\end{equation}
Introducing the basis $\oC^2=\langle\rrho,\vvarphi\rangle$ for the first summand and $\oC^3=\langle\bkappa_0,\bkappa_1,\bkappa_2\rangle$ for the second, the $\SLiiZ$ action is
\begin{align}\label{C2}
\modS(\rrho)&= -\rmi \vvarphi\ ,& \modS(\vvarphi) &= \rmi \rrho\ ,\\
\modT(\rrho) &= b \rrho\ , 
&\modT(\vvarphi) &= b (\vvarphi - \rmi\beta^{-2} \rrho)\ ,
\end{align}
and 
\begin{align}
\modS(\bkappa_0)&= \half(\bkappa_0 - 2\bkappa_1 + \bkappa_2)\ ,& \modS(\bkappa_1) &= \half(\bkappa_2 - \bkappa_0)\ , & \modS(\bkappa_2) &=\half(\bkappa_0 + 2\bkappa_1 + \bkappa_2)\ ,\label{C3-S}\\
\modT(\bkappa_0) &= -b\beta \bkappa_0\ , &\modT(\bkappa_1) &= b \bkappa_1\ , 
&\modT(\bkappa_2) &= b\beta \bkappa_2\ , \label{C3}
\end{align}
with $b$ as in \eqref{eq:projectivity-factor}.
At $\beta=e^{-\rmi\pi/4}$, this action is equivalent to the ``standard'' 
$\SLiiZ$-representation in~\cite{[FGST]}.
\end{Prop}
\begin{proof}
We set for the basis in the $\oC^3$ component in~\eqref{Z-decomp}
\begin{equation}
\bkappa_0 = \cchi^-_2,\qquad\bkappa_1 = \cchi^+_1 + \cchi^-_1, \qquad\bkappa_2 = \cchi^+_2 
\end{equation}
and for the $\oC^2$ component:
\begin{equation}
\rrho = \half(\cchi^+_1 - \cchi^-_1),\qquad
\vvarphi  = \ffrac{\rmi}{2}(\radmap^+_1-\radmap^-_1).
\end{equation}
For the $\oC^3$ component, it is then a simple check that the $\modS$-action is given by~\eqref{C3-S}.
The first part of~\eqref{C2} is obvious while during the calculation of the second equality in~\eqref{C2}, it is important to note that $\modS^2$
 acts as the identity on the centre $\cZ$ of~$\Q$,  i.e., 
 $\modS^2|_{\cZ}=\id$~\footnote{For a factorisable Hopf algebra,  $\modS^2$ acts via 
the antipode~\cite{[LM]}.
Since for $\UresSL2$ the antipode acts as the identity on the centre, so does $\modS^2$.
 This general property was actually used in~\cite{[FGST]} during the calculation of $\modS$. In our  context of  \textsl{quasi}-Hopf algebras, we are not aware of this property in general and should thus give a direct argument.  We express 
 $\vvarphi=-\rmi\beta^2\drmap(\gamma(1))$, where $\gamma(1)=\mathrm{Tr}_{\PP^+_1\oplus\PP^-_1}(
 \balance^{-1}-\sigma_1
 )$ is the ``pseudo-trace'' $q$-character 
from~\cite{[GT]} and $\sigma_1:\PP^+_1\oplus\PP^-_1\to\PP^+_1\oplus\PP^-_1$ is a linear map defined in~\cite[eq.~(3.10) with $a_0=1$]{[GT]}. Then we compute $\radmap(\gamma(1))=\beta^2\idem_0$ and therefore $\modS(\vvarphi)=\radmap\circ\drmap^{-1}(\vvarphi) = -\rmi\beta^2\radmap(\gamma(1))=\rmi\rrho$. And so $\modS^2$ acts on $\cZ$ by the identity for any value of $\beta$ indeed.}.
 The ribbon element decomposition~\eqref{ribbon-decomp} simplifies the calculation of the $\modT$ action as well. 
In addition to $\modS^2=\one$ one also easily checks the relation $(\modS\modT)^3=\one$.

The mapping 
 \begin{equation*}
 \rrho\mapsto\rrho(1),\qquad \vvarphi\mapsto\vvarphi(1), \qquad \bkappa_s\mapsto\bkappa(s),\quad s=0,1,2,
 \end{equation*}
 between our basis and the one in~\cite[Sec.~5]{[FGST]} for $p=2$ establishes the equivalence 
 at (the symplectic fermions value of) $\beta=e^{-\rmi\pi/4}$. 
\end{proof}

Looking at the eigenvalues of $\modT$ in Proposition \bref{prop:sl2z-action} we see that different values of $\beta^2$ give inequivalent representations of $\SLiiZ$. On the other hand, changing $\beta$ to $-\beta$ gives an equivalent representation (via $\bkappa_0 \leftrightarrow \bkappa_2$ and $\bkappa_1 \to -\bkappa_1$).
We have thus obtained two inequivalent $\SLiiZ$ actions on $\cZ$ parametrised by $\beta^2$.

\newcommand\arxiv[2]      {\href{http://arXiv.org/abs/#1}{#2}}
\newcommand\doi[2]        {\href{http://dx.doi.org/#1}{#2}}
\newcommand\httpurl[2]    {\href{http://#1}{#2}}

\end{document}